\documentclass[10pt,letterpaper,reqno]{amsart}
\usepackage{amsfonts, amsmath, amssymb, amscd, amsthm, graphicx,enumerate}
\usepackage{color}
\usepackage{hyperref}

\makeatletter
\def\blfootnote{\xdef\@thefnmark{}\@footnotetext}
\makeatother

\hypersetup{colorlinks=true,linkcolor=red,citecolor=blue,urlcolor=blue}

\newtheorem{thm}{Theorem}[section]
\newtheorem{mainthm}{Theorem}

\newcommand{\restatedmainthmname}{}
\newtheorem*{mainthmrestatedbase}{\restatedmainthmname}
\newenvironment{mainthmrestated}[2]
 {\renewcommand{\restatedmainthmname}{Theorem~\ref{#1} (#2)}%
  \begin{mainthmrestatedbase}}
 {\end{mainthmrestatedbase}}
\newcommand{\restatedcorname}{}
\newtheorem*{correstatedbase}{\restatedcorname}
\newenvironment{correstated}[2]
 {\renewcommand{\restatedcorname}{Corollary~\ref{#1} (#2)}%
  \begin{correstatedbase}}
 {\end{correstatedbase}}
\newtheorem{cor}[thm]{Corollary}
\newtheorem{lem}[thm]{Lemma}
\newtheorem{prop}[thm]{Proposition}

\theoremstyle{definition}
\newtheorem{defn}[thm]{Definition}
\theoremstyle{remark}
\newtheorem{rem}[thm]{Remark}
\newtheorem{ex}[thm]{Example}

\newtheorem{conv}[thm]{Convention}

\renewcommand{\phi}{\varphi}
\renewcommand{\epsilon}{\varepsilon}

\newcommand{\R}{\mathbb R}
\newcommand{\Z}{\mathbb Z}
\renewcommand{\Pr}{\mathbb P}
\newcommand{\Aut}{\operatorname{Aut}}
\newcommand{\Out}{\operatorname{Out}}
\newcommand{\Stab}{\operatorname{Stab}}

\newcommand{\rank}{\operatorname{rank}}
\newcommand{\supp}{\operatorname{supp}}
\newcommand{\Curr}{\operatorname{Curr}}
\newcommand{\PCurr}{\mathbb P\Curr}
\newcommand{\cv}{\mathrm{cv}}
\newcommand{\cvbar}{\overline{\mathrm{cv}}}
\newcommand{\dd}{\partial^2}

\newcommand{\BBT}{\operatorname{BBT}}
\newcommand{\Lip}{\operatorname{Lip}}
\newcommand{\diam}{\operatorname{diam}}
\newcommand{\Ov}{\operatorname{Ov}}
\newcommand{\Sym}{\operatorname{Sym}}

\newcommand{\vol}{\operatorname{vol}}
\newcommand{\Normal}[1]{\left\langle\!\left\langle #1\right\rangle\!\right\rangle}
\newcommand{\calK}{\mathcal K}
\newcommand{\calR}{\mathcal R}

\newcommand{\calQ}{\mathcal Q}
\newcommand{\ol}{\overline}

\newcommand{\rhoSigned}{\rho_{\mathrm{per}}^{\pm}}
\newcommand{\DeltaA}{\Delta_A}

\newcommand{\sh}{\operatorname{sh}}
\newcommand{\cdim}{\operatorname{cd}}
\newcommand{\Rel}{\operatorname{Rel}}
\newcommand{\CR}{\mathcal C}
\newcommand{\FR}{\mathcal F}

\begin{document}
\raggedbottom

\title[Isomorphism Rigidity for for Generic Finitely Presented Groups]{Small Cancellation Stability and Isomorphism Rigidity for Generic Finitely Presented Groups}

\author[I. Kapovich]{Ilya Kapovich}

\address{Department of Mathematics and Statistics, Hunter College of CUNY\newline
  \indent 695 Park Ave, New York, NY 10065
  \newline \indent  {\url{http://math.hunter.cuny.edu/ilyakapo/}}}
  \email{ik535@hunter.cuny.edu}


\date{}
\hypersetup{
  pdftitle={Small Cancellation Stability and Isomorphism Rigidity for Generic Finitely Presented Groups},
  pdfauthor={Ilya Kapovich},
  pdfsubject={The Kapovich--Schupp Stability Conjecture, random cyclically reduced relators, automorphic small cancellation, and isomorphism rigidity},
  pdfkeywords={free groups, random cyclically reduced words, lambda-stability, small cancellation, Whitehead algorithm, isomorphism rigidity}
}

\begin{abstract}
Let $F_m=F(a_1,\dots,a_m)$ with $m\ge 2$, and fix $q\ge 1$.  For every fixed $0<\lambda<1$, we prove that a $q$-tuple $\mathbf W_n$ of independent uniformly random cyclically reduced words of length $n$ is \emph{$\lambda$-stable} with probability converging to $1$ exponentially fast. Namely, for every $\Phi\in \Aut(F_m)$, the tuple $\Phi(\mathbf W_n)$, after cyclic reduction and symmetrization, satisfies the $C'(\lambda)$ small cancellation condition.

Combining generic $\lambda$-stability with Greendlinger normal-closure rigidity and with previous results of Kapovich-Schupp-Shpilrain on generic Nielsen uniqueness and generic Whitehead rigidity we establish, for any fixed $m\ge 2, q\ge 1$, isomorphism rigidity for generic $m$-generator $q$-relator groups. Thus we show that two such generic groups $\langle a_1,\dots, a_m| r_1,\dots, r_q\rangle$ and $\langle a_1,\dots, a_m| s_1,\dots, s_q\rangle$ are isomorphic if and only if, after possibly permuting and inverting the generators $a_1,\dots, a_m$, the relator tuples $(r_1,\dots, r_q)$ and $(s_1,\dots, s_q)$ are the same, up to reordering, cyclic permutations and inverting the relators. Among the applications, we obtain a quadratic-time algorithm that generically solves the isomorphism problem for $m$-generator $q$-relator groups, and  show that the number of isomorphism types represented by $m$-generator $q$-relator presentations with cyclically reduced relators of length $n$ is asymptotic to
\[
 \frac{(2m-1)^{qn}}{2^{m+q}m!\,q!\,n^q}.
\]
 
The proof of generic $\lambda$-stability relies on the use of geodesic currents and on our deterministic sufficient criterion for a $q$-tuple $\mathbb W$ in $F_m$ to be $\lambda$-stable in terms of the components of $\mathbb W$ being sufficiently projectively close to filling currents. 
\end{abstract}

\maketitle
\begingroup
\setlength{\parskip}{0pt}
\setcounter{tocdepth}{1}
\tableofcontents
\endgroup

\section{Introduction}

\subsection{Uniform automorphic small cancellation}

Random and generic phenomena play a central role in geometric and combinatorial group theory.  Gromov's density and few-relator models initiated the study of groups with long random relations~\cite{Gro87,Gro93}. Several different models of genericity for finitely presented groups have been studied, see, for example, \cite{Ol92,Cha95,AO96,Oll05}, with word-hyperbolicity being a persistent common feature of these constructions. Random-group methods have also produced striking applications beyond hyperbolicity.  Gromov used a random-group construction based on expander graphs to produce finitely generated groups which do not admit a uniform embedding into a Hilbert space~\cite{Gro03}.  Random groups also provide important sources of Kazhdan groups: in particular, for $1/3<d<1/2$, a random group in Gromov's density model at density $d$ has property~(T) with probability tending to one~\cite{Zuk03,KK13}.  Random groups in the few relator model exhibit other interesting properties, such as bounded rank local freeness and Nielsen uniqueness, ~\cite{AO96,Arz97,KS05}, with applications to isomorphism rigidity for random one-relator groups~\cite{KSS06,KS09}.

Throughout the paper,
\[
 F_m=F(A)=F(a_1,\dots,a_m),\qquad m\ge 2,\quad  A=\{a_1,\dots, a_m\}. 
\]
Whenever $X_n$ is a sequence of nonempty finite sampling spaces, we say that subsets $Y_n\subseteq X_n$ are \emph{exponentially generic} if there are constants $C,c>0$ such that
\[
 1-\frac{|Y_n|}{|X_n|}\le Ce^{-cn}
\]
for every $n\ge 1$.  Likewise, a sequence of events $E_n$ occurs \emph{with exponentially high probability} if $\Pr(E_n)\ge 1-Ce^{-cn}$ for some $C,c>0$ and every $n\ge 1$.  

The \emph{isomorphism-rigidity} theorem of Kapovich-Schupp-Shpilrain for random one-relator groups in~\cite{KSS06} says, roughly, that one-relator groups $G_r=\langle a_1,\dots, a_m|r\rangle$ and $G_s=\langle a_1,\dots, a_m|s\rangle$, with long "random" cyclically reduced defining relator $r,s\in F_m$, are isomorphic as groups if and only if there exists a ``relabeling" automorphism $\Theta$ of $F_m$ permuting and possibly inverting the generators $A$, such that $r$ is a cyclic permutation of $\Theta(s^{\pm 1})$. The proof of this result in \cite{KSS06} combined three main ingredients.  First, the Nielsen-uniqueness theorem of Kapovich--Schupp~\cite{KS05} says that, for fixed $m\ge 2$ and $q\ge 1$, there is an exponentially generic class of $m$-generator $q$-relator presentations whose groups have a single Nielsen class of $m$-tuples generating nonfree subgroups.  Second, Kapovich--Schupp--Shpilrain~\cite{KSS06} gave a generic-case analysis of Whitehead's algorithm in $F_m$: generic cyclically reduced words satisfy a strong Whitehead-minimality condition, and automorphic equivalence between such words reduces to relabeling and cyclic permutation.  Third, Magnus's classical normal-closure theorem~\cite{Mag30} says that elements $r,s\in F_m$ have the same normal closure if and only if $r$ is conjugate in $F_m$ to $s^{\pm1}$.  In the one-relator case, these three inputs turn an abstract group isomorphism into the explicit relation between relators needed for isomorphism rigidity.

For $q\ge 2$ defining relators, the analogous conclusion of Magnus's theorem is false for relator sets in general, and this fact was the main obstacle to extending the results of~\cite{KSS06} to random multi-relator quotients of free groups.  Kapovich--Schupp overcame the corresponding difficulty in~\cite{KS09} for random $q$-relator quotients of the modular group $M=\Z_2\ast\Z_3$.  There the finite orders of the standard generators and the restricted Nielsen theory for generating pairs provide enough additional control to combine small cancellation with a Greendlinger-type normal-closure argument.  The analogous generic multi-relator rigidity problem for quotients of $F_m$ remained open.  In the present paper, we overcome this difficulty by proving $\lambda$-stability for random $q$-tuples of relators in $F_m$, for every $q\ge 1$, and combining it with Greendlinger's theorem~\cite{Gre61} for symmetrized small-cancellation sets having the same normal closure.  This supplies the required multi-relator substitute for Magnus's theorem; the other inputs are the Nielsen-uniqueness results of~\cite{KS05} and the generic Whitehead-minimality and stabilizer results of~\cite{KSS06}.

The missing ingredient is therefore a normal-closure uniqueness principle robust enough to survive arbitrary automorphisms of $F_m$.  This situation leads naturally to the stability problem studied here: whether small cancellation holds uniformly throughout the automorphic orbit of a random relator tuple.  The order of quantifiers is crucially important. For any fixed $\Phi\in \Aut(F_m)$ standard quasi-isometry considerations imply that, for any fixed $0<\lambda<1$, there exists $0<\lambda'<\lambda$ such that a sufficiently long cyclically reduced $C'(\lambda')$-word $w\in F_m$ has the property that the cyclic reduction of $\Phi(w)$ satisfies $C'(\lambda)$, after symmetrization. Since for any $0<\lambda'<1$, random cyclically reduced $w_n\in F_m$ satisfy $C'(\lambda')$ for all large enough $n$,  it follows that the cyclic reduction of $\Phi(w_n)$ satisfies $C'(\lambda)$, after symmetrization. However, requiring that a random cyclically reduced word satisfies this conclusion \emph{for all} $\Phi\in \Aut(F_m)$ makes the problem much more difficult. 

For $g\ne1$, let $\|g\|_A$ be the cyclically reduced length of its conjugacy class.  Recall informally that a nontrivial element is \emph{root-free} if it is not a proper power, and a tuple is root-free if each of its entries is root-free; a tuple is \emph{irredundant} if no two distinct entries are conjugate to one another, even after inverting one of them.

In this paper we introduce and study the following new key notion.  For $0<\lambda<1$, we call a tuple of cyclically reduced words in $F_m$ \emph{$\lambda$-stable} if, after applying any automorphism of $F_m$ and cyclically reducing, its symmetrization satisfies $C'(\lambda)$.  See Subsection~\ref{subsec:free-small-prelim} for the precise definitions.

\begin{rem}[Small cancellation versus stability]\label{rem:small-cancellation-not-stable}
Ordinary $C'(\lambda)$ small cancellation is far weaker than $\lambda$-stability.  If $m\ge3$, $0<\lambda\le1/6$, and $W_n$ is a uniform positive word of length $n$ in $F(a_2,\dots,a_m)$, then $a_1W_n$ satisfies $C'(\lambda)$ with exponentially high probability but is primitive, hence not $\lambda$-stable.  Indeed, the automorphic orbit of every primitive element contains $a_1a_2^N$, whose symmetrization fails $C'(\lambda)$ for all sufficiently large $N$.  The phenomenon already occurs in rank two: the positive word
\[
 a^5ba^3b^2ababa^2b^4ab^3
\]
has length $25$, is root-free, and satisfies $C'(1/6)$, but its image under $a\mapsto ab$, $b\mapsto b$ has two distinct cyclic conjugates with common prefix $(ab)^4$ of length $8>38/6$.
\end{rem}

For a root-free irredundant tuple $\mathbf g=(g_1,\dots,g_q)$, let $\Delta_A(\mathbf g)$ denote its \emph{supremal automorphic piece ratio}.  For each $\phi\in\Out(F_m)$, choose an automorphism representing $\phi$, cyclically reduce the images of the components of $\mathbf g$, and symmetrize the resulting tuple by taking all cyclic shifts of these words and their inverses.  For every ordered pair of distinct words $r,r'$ in this symmetrization, let $v$ be their maximal common initial segment and record the ratio
\[
 \frac{|v|_A}{|r|_A}.
\]
Then $\Delta_A(\mathbf g)$ is the supremum of these ratios over all such ordered pairs and all $\phi\in\Out(F_m)$, with value $0$ when no pieces occur; if $\mathbf g$ is not root-free and irredundant, we put $\Delta_A(\mathbf g)=\infty$.  Thus $\Delta_A(\mathbf g)$ measures the worst relative piece overlap occurring anywhere in the automorphic orbit of $\mathbf g$, and $\Delta_A(\mathbf g)<\lambda$ implies that $\mathbf g$ is $\lambda$-stable.  See Definition~\ref{def:admissible} for the formal definition.

For $n\ge 1$, let $\CR_n$ and $\FR_n$ be the sets of cyclically and freely reduced words of length $n$ over $A^{\pm1}$, respectively.  Unless stated otherwise, random choices are uniform.  For a positive sequence $a_n$, we write $Y_n=O_{\Pr}(a_n)$ if $Y_n/a_n$ is bounded in probability; explicitly, for every $\delta>0$ there is $K>0$ such that $\Pr(|Y_n|>Ka_n)<\delta$ for all sufficiently large $n$.

\begin{mainthm}[Uniform $\lambda$-stability]\label{mainthm:stability}
Fix $m\ge 2$, $q\ge 1$, and $0<\lambda<1$.  Let
\[
 \mathbf C_n=(C_{1,n},\dots,C_{q,n})
\]
be an independent uniformly random $q$-tuple in $\CR_n^q$.  Then the following hold.
\begin{enumerate}[(1)]
\item There are constants $C,c>0$ such that, for every $n\ge 1$, with probability at least $1-Ce^{-cn}$ the tuple $\mathbf C_n$ is root-free, irredundant, and $\lambda$-stable.

\item More precisely,
\[
 \Delta_A(\mathbf C_n)=O_{\Pr}\left(\frac{\log n}{n}\right),
\]
and for every fixed $\epsilon>0$ there are $C_\epsilon,c_\epsilon>0$ such that
\[
 \Pr\bigl(\Delta_A(\mathbf C_n)>\epsilon\bigr)
 \le C_\epsilon e^{-c_\epsilon n}
\]
for every $n\ge 1$.

\item If $W_{1,n},\dots,W_{q,n}$ are independent uniformly random words in $\FR_n$ and $\widehat W_{i,n}$ is the cyclically reduced form of $W_{i,n}$, then the conclusions of parts~\textup{(1)}--\textup{(2)} hold for
\[
 \widehat{\mathbf W}_n=(\widehat W_{1,n},\dots,\widehat W_{q,n});
\]
in particular,
\[
 \Delta_A(\widehat{\mathbf W}_n)
 =O_{\Pr}\left(\frac{\log n}{n}\right),
\]
and the corresponding fixed-scale tails are exponential.
\end{enumerate}
\end{mainthm}

\begin{cor}[Freely reduced sphere and ball models]\label{cor:sphere-ball-stability}
Fix $m\ge 2$, $q\ge 1$, and $0<\lambda<1$.
\begin{enumerate}[(1)]
\item Let
\[
 \mathbf W_n=(W_{1,n},\dots,W_{q,n})
\]
be an independent uniformly random $q$-tuple in $\FR_n^q$.  Then there are constants $C,c>0$ such that, for every $n\ge 1$, with probability at least $1-Ce^{-cn}$ the tuple $\mathbf W_n$ is root-free, irredundant, and $\lambda$-stable.  Moreover,
\[
 \Delta_A(\mathbf W_n)=O_{\Pr}\left(\frac{\log n}{n}\right),
\]
and for every fixed $\epsilon>0$ there are $C_\epsilon,c_\epsilon>0$ such that
\[
 \Pr\bigl(\Delta_A(\mathbf W_n)>\epsilon\bigr)
 \le C_\epsilon e^{-c_\epsilon n}
\]
for every $n\ge 1$.

\item Let
\[
 B_A(n)=\{g\in F_m:|g|_A\le n\},
\]
and let $\mathbf U_n=(U_{1,n},\dots,U_{q,n})$ be uniformly distributed on $B_A(n)^q$ (equivalently, its components are independent uniformly random freely reduced words of length at most $n$).  Then the conclusions of part~\textup{(1)} hold with $\mathbf U_n$ in place of $\mathbf W_n$.
\end{enumerate}
\end{cor}

The deterministic core underlying these probabilistic conclusions is a finite-frequency criterion.  For a nontrivial freely reduced word $v$ and a nontrivial cyclically reduced word $W$, let $(v,W)_A$ be the number of nonsymmetrized cyclic occurrences of $v$ in $W$, and put
\[
 \langle v,W\rangle_A=(v,W)_A+(v^{-1},W)_A=(v^{\pm1},W)_A.
\]
Let $T_A$ be the unit-edge Cayley tree of $F_m$ with respect to $A$, and let $\partial F_m$ be its Gromov boundary.  A \emph{geodesic current} on $F_m$ is a positive locally finite Borel measure on $\partial^2F_m=(\partial F_m\times\partial F_m)\setminus\{(\xi,\xi)\}$ that is invariant under the diagonal $F_m$-action and the flip $(\xi,\zeta)\mapsto(\zeta,\xi)$.  For such a current $\nu$ and a nontrivial freely reduced word $v$, the cylinder coordinate $\langle v,\nu\rangle_A$ is the $\nu$-mass of the set of oriented geodesics in $T_A$ containing the oriented segment $[1,v]$.  The counting current $\eta_W$ associated to a nontrivial cyclically reduced word $W$ is characterized by $\langle v,\eta_W\rangle_A=\langle v,W\rangle_A$.  The standard tree--current intersection form is denoted $\langle T,\nu\rangle$; the current $\nu$ is \emph{filling} if this quantity is positive for every nontrivial tree $T\in\cvbar_m$, the length-function closure of Outer space.  These notions are recalled in detail in Subsections~\ref{subsec:outer-prelim}--\ref{subsec:currents-prelim}.

\begin{thm}[Finite-frequency stability criterion]\label{thm:finite-frequency-stability}
Fix $q\ge 1$ and let $\nu_1,\dots,\nu_q$ be filling geodesic currents on $F_m$ satisfying
\[
 \langle T_A,\nu_i\rangle=1
 \qquad(1\le i\le q),
\]
and let $0<\lambda<1$.  Then there exist constants $0<\epsilon<1$ and $0<\lambda'<\lambda$, and integers $M,n_0\ge 1$, such that the following holds.  Let
\[
 \mathbf W=(W_1,\dots,W_q)
\]
be a root-free irredundant tuple of nontrivial cyclically reduced words, not necessarily of equal lengths, and put
\[
 \ell_i=\|W_i\|_A=|W_i|_A
 \qquad(1\le i\le q).
\]
Suppose that $\ell_i\ge n_0$ for every $i$, that the symmetrization of $\mathbf W$ satisfies $C'(\lambda')$, and that, for every $1\le i\le q$ and every nontrivial freely reduced word $v$ with $|v|_A\le M$,
\[
 \left|
 \frac{\langle v,W_i\rangle_A}{\ell_i}
 -\langle v,\nu_i\rangle_A
 \right|<\epsilon.
\]
Then
\[
 \Delta_A(\mathbf W)<\lambda.
\]
In particular, $\mathbf W$ is $\lambda$-stable.
\end{thm}

Thus, for root-free irredundant tuples, a property quantified over all automorphisms of $F_m$ follows from ordinary small cancellation together with finitely many local frequency inequalities normalized by the individual relator lengths.  Theorem~\ref{thm:finite-frequency-stability} is proved in Section~\ref{sec:abstract-stability}.

We prove Theorem~\ref{mainthm:stability} in Section~\ref{sec:random-reduced}, followed by a restatement and proof of Corollary~\ref{cor:sphere-ball-stability}. Theorem~\ref{mainthm:stability} serves as the main tool to establish subsequent probabilistic results. The proof of Theorem~\ref{mainthm:stability} combines geometric arguments regarding Outer space, bounded backtracking and optimal Lipschitz distortion with arguments involving geodesic currents and the Kapovich-Lustig intersection form~\cite{KL09,KL10}. 

Proposition~\ref{prop:deterministic-piece} provides a key deterministic bound on the length of a piece $v$ for the symmetrized closure of $\phi(\mathbf W)$, where $\mathbf W$ is an arbitrary irredundant root-free $q$-tuple of cyclically reduced words in $F_m$:
\[
 |v|\le2\Lambda_A(\phi)\bigl(\rho_{\mathrm{per}}^{\pm}(\mathbf W)+2m\bigr).
\]
Here $\rho_{\mathrm{per}}^{\pm}(\mathbf W)$ records a geometrically defined upper bound for the maximum piece length for the symmetrization of (the cyclically reduced form of) $\mathbf W$, and $\Lambda_A(\phi)$ is the optimal Lipschitz constant of an $F_m$-equivariant map $T_A\to T_A\phi$, where $T_A$ is the Cayley tree of $F_m=F(A)$ with respect to $A$.  
Lemma~{lem:uniform-filling-neighborhood} shows that for a suitably chosen neigborhood $U$ of a filling current $\nu$, every current $\eta\in U$ has a controlled distortion under the action of arbitrary $\phi\in \Out(F_m)$, namely
\[
\langle T_A,\phi\eta\rangle \ge c \Lambda_A(\phi)
\]
where $c>0$ is some constant independent of $\phi$ and $\eta$. In particular, this inequality applies to ``rational" currents in $U$ corresponding to conjugacy classes in $F_m$. Combining these tools ultimately allows us to deduce Theorem~\ref{thm:finite-frequency-stability}.

Section~\ref{sec:abstract-stability} first proves the deterministic finite-frequency criterion and then derives the probabilistic master estimate from finite-coordinate concentration and signed-overlap tails.  Section~\ref{sec:positive-corollary} then recovers $\lambda$-stability for independent positive Bernoulli words, allowing different strictly positive probability vectors among components, and records an almost-sure extension to finitely supported group random walks; see Corollary~\ref{cor:finitely-supported-random-walk}.

\subsection{Generic multi-relator rigidity}

The preceding stability theorem, combined with Greendlinger's normal-closure theorem and the Nielsen and Whitehead inputs described above, yields the generic multi-relator rigidity theorem.  Let $\Rel(A)$ denote the finite group of signed relabeling automorphisms of $F_m$, that is, automorphisms that permute $A^{\pm1}$ and commute with inversion.  Thus
\[
 |\Rel(A)|=2^m m!.
\]

\begin{mainthm}[Generic few-relator rigidity]\label{mainthm:rigidity}
Fix $m\ge 2$ and $q\ge 1$.  For a finite tuple $R=(r_1,\dots,r_s)$ of elements of $F_m$, write
\[
 G_R=\langle a_1,\dots,a_m\mid r_1,\dots,r_s\rangle,
\]
and denote by $\ol a_i$ the image of $a_i$ in $G_R$.
Here and below, ``finitely specified'' means that, for fixed $m,q$, finite rational auxiliary data are chosen once and for all; no uniform procedure producing these data from $m$ and $q$ is asserted.
There are finitely specified subsets
\[
 \calQ_{m,q}(n)\subseteq\CR_n^q
 \qquad(n\ge 1)
\]
and constants $C,c>0$ such that:
\begin{enumerate}[(1)]
\item For every $n\ge 1$, a uniform random tuple in $\CR_n^q$ belongs to $\calQ_{m,q}(n)$ with probability at least $1-Ce^{-cn}$.  Membership in $\calQ_{m,q}(n)$ is decidable in deterministic polynomial time in $n$.
\item Every $q$-tuple $R\in\calQ_{m,q}(n)$ is irredundant, root-free, and $1/6$-stable.
\item If $R=(r_1,\dots,r_q)\in\calQ_{m,q}(n)$, then $G_R$ is torsion-free, one-ended, and non-elementary hyperbolic; every subgroup generated by at most $m-1$ elements is free; every $m$-tuple generating a nonfree subgroup is Nielsen-equivalent to the standard tuple $(\ol a_1,\dots,\ol a_m)$; and $G_R$ is co-Hopfian, meaning that every injective endomorphism of $G_R$ is surjective.
\item If $R=(r_1,\dots,r_q)\in\calQ_{m,q}(n)$, then $G_R$ is complete, meaning that its center is trivial and every automorphism of $G_R$ is inner.  Its boundary is one-dimensional, connected, and has no local cut points, and hence is homeomorphic to the Menger curve or the Sierpi\'nski carpet.  If $q\ge m-1$, then $\partial G_R$ is the Menger curve.
\item Let $R=(r_1,\dots,r_q)\in\calQ_{m,q}(n)$, and let $S=(s_1,\dots,s_t)$ be an irredundant tuple of nontrivial elements of $F_m$ whose symmetrization of cyclically reduced representatives satisfies $C'(1/6)$.  Then $G_R\cong G_S$ if and only if $t=q$ and there are $\Phi\in\Aut(F_m)$, $\sigma\in\Sym(q)$, signs $\epsilon_i\in\{\pm1\}$, and $g_i\in F_m$ such that
\[
 \Phi(r_i)=g_i s_{\sigma(i)}^{\epsilon_i}g_i^{-1}
 \qquad(1\le i\le q).
\]
\item If $R=(r_1,\dots,r_q)\in\calQ_{m,q}(n)$ and $S=(s_1,\dots,s_q)\in\calQ_{m,q}(n')$, then $G_R\cong G_S$ if and only if $n=n'$ and there are $\Theta\in\Rel(A)$, $\sigma\in\Sym(q)$, and signs $\epsilon_i\in\{\pm1\}$ such that, for each $i=1,\dots,q$, the word $s_i$ is a cyclic permutation of $\Theta(r_{\sigma(i)})^{\epsilon_i}$.  \end{enumerate}
\end{mainthm}

\begin{rem}[Stability Conjecture]\label{rem:KS-stability-conjecture}
The Stability Conjecture of Kapovich--Schupp~\cite[Conjecture~1.2]{KS05b} says the following.  Fix $m\ge 2$ and $q\ge 1$.  There exists an algorithmically recognizable generic family $\mathcal D_n\subseteq  \CR_n^q$, $n\ge 1$, with the following property.  If $\boldsymbol\sigma\in \mathcal D_n, \boldsymbol\tau\in \mathcal D_{n'}^q$ and $\Phi\in\Aut(F_m)$ satisfy
\[
 N_{\calR_A(\boldsymbol\sigma)}
 =N_{\calR_A(\Phi(\boldsymbol\tau))},
\]
then
\[
 \calR_A(\boldsymbol\sigma)
 =\calR_A(\Phi(\boldsymbol\tau)).
\]
Theorem~\ref{mainthm:rigidity} directly gives this conclusion in the present few-relator model, in an exponentially generic and polynomial-time recognizable form.  Indeed, part~\textup{(1)} gives the exponentially generic recognizable classes $\calQ_{m,q}(n)$, while part~\textup{(2)} says that every tuple in these classes is irredundant, root-free, and $1/6$-stable.  Thus, if
\[
 \boldsymbol\sigma\in\calQ_{m,q}(n),
 \qquad
 \boldsymbol\tau\in\calQ_{m,q}(n')
\]
for some $n,n'\ge 1$ and $\Phi\in\Aut(F_m)$, then both $\calR_A(\boldsymbol\sigma)$ and $\calR_A(\Phi(\boldsymbol\tau))$ satisfy $C'(1/6)$.  If these two symmetrized sets have the same normal closure, Theorem~\ref{thm:Greendlinger-normal} gives their equality.  Hence Theorem~\ref{mainthm:rigidity} directly implies the Stability Conjecture (in fact with exponential genericity in the model used here).
\end{rem}

\begin{rem}[Labelled Cayley graphs]\label{rem:labelled-cayley-graphs}
Part~\textup{(6)} has the following equivalent labelled-graph formulation.  If
\[
 R\in\calQ_{m,q}(n),\qquad S\in\calQ_{m,q}(n'),
\]
then $G_R\cong G_S$ if and only if there are $\Theta\in\Rel(A)$ and a graph isomorphism
\[
 \operatorname{Cay}(G_R,A)\longrightarrow \operatorname{Cay}(G_S,A)
\]
that sends every oriented edge labelled $a\in A^{\pm1}$ to an oriented edge labelled $\Theta(a)$.  Equivalently, after relabelling the edge labels of the target Cayley graph by $\Theta^{-1}$, the two Cayley graphs are isomorphic as $A^{\pm1}$-labelled graphs.
\end{rem}

Section~\ref{sec:generic-class} proves Theorem~\ref{mainthm:rigidity}, using the Greendlinger normal-closure theorem from Section~\ref{sec:normal-closure}.

\subsection{Algorithms and enumeration}

The same rigidity description yields efficient isomorphism procedures on the generic class and an asymptotic enumeration of the resulting isomorphism types.

\begin{mainthm}[Isomorphism algorithms]\label{mainthm:algorithms}
Fix $m\ge 2$ and $q\ge 1$.  For a finite tuple $U=(u_1,\dots,u_s)$ of elements of $F_m$, write
\[
 G_U=\langle a_1,\dots,a_m\mid u_1,\dots,u_s\rangle.
\]
\begin{enumerate}[(1)]
\item Fix $R=(r_1,\dots,r_q)\in\calQ_{m,q}(n_0)$.  There is a quadratic-time algorithm for the following restricted isomorphism problem.  The input is a finite tuple $S=(s_1,\dots,s_t)$ of nontrivial cyclically reduced words known to be irredundant and to have symmetrization satisfying $C'(1/6)$; the algorithm decides whether $G_R\cong G_S$.  With
\[
 N=\sum_{j=1}^t|s_j|_A,
\]
the running time is $O(N^2)$.  (The assumption on $S$ can be checked within the same $O(N^2)$ bound; an input failing it is reported as outside the restricted problem.)
\item There is a linear-time algorithm which, given tuples $R\in\CR_n^q$ and $S\in\CR_{n'}^q$ known to belong to $\calQ_{m,q}(n)$ and $\calQ_{m,q}(n')$, respectively, decides whether $G_R\cong G_S$.  The algorithm's running time is linear in $\max\{n,n'\}$.  The algorithm does not test the assumptions on $R,S$.
\end{enumerate}
\end{mainthm}

We prove Theorem~\ref{mainthm:algorithms} in Subsection~\ref{subsec:algorithms}.

\begin{mainthm}[Asymptotic number of isomorphism types]\label{mainthm:counting}
Fix $m\ge 2$ and $q\ge 1$.  For $n\ge 1$ let $I_{m,q}^{\mathrm{cr}}(n)$ be the number of group isomorphism types represented by presentations
\[
 \langle a_1,\dots, a_m\mid r_1,\dots,r_q\rangle,
\]
where $r_1,\dots, r_q$ are cyclically reduced words of length $n$ in $F_m=F(a_1,\dots, a_m)$. 
Then
\[
 I_{m,q}^{\mathrm{cr}}(n)
 \sim
 \frac{(2m-1)^{qn}}{2^{m+q}m!\,q!\,n^q}
 \qquad(n\to\infty).
\]
\end{mainthm}
Here $\sim$ means that the ratio of the two quantities goes to $1$ as $n\to\infty$.
We prove Theorem~\ref{mainthm:counting} in Subsection~\ref{subsec:counting}.  For $q=1$, the conclusion of Theorem~\ref{mainthm:counting}  agrees with the precise one-relator asymptotic of Kapovich--Schupp~\cite{KS05b}.

\subsection{Organization}

Section~\ref{sec:preliminaries} gives the preliminaries.  Sections~\ref{sec:signed-overlap}--\ref{sec:abstract-stability} establish the deterministic overlap estimate, the finite-frequency stability criterion, and its probabilistic consequences.  Section~\ref{sec:random-reduced} proves Theorem~\ref{mainthm:stability} and Corollary~\ref{cor:sphere-ball-stability}, and Section~\ref{sec:positive-corollary} treats positive Bernoulli words and records an almost-sure extension to finitely supported group random walks.  Section~\ref{sec:generic-input} records the Nielsen and Whitehead inputs.  Section~\ref{sec:normal-closure} proves normal-closure rigidity.  Sections~\ref{sec:isomorphism}--\ref{sec:generic-class} prove Theorem~\ref{mainthm:rigidity}, and Section~\ref{sec:algorithms-counting} gives the algorithms and enumeration.

\section{Preliminaries}\label{sec:preliminaries}

We collect notation and standard facts used below.

\subsection{Free groups, cyclic words, and small cancellation}\label{subsec:free-small-prelim}

\begin{conv}[Standing free-group notation]\label{conv:standing}
Unless explicitly stated otherwise, $m\ge 2$, $F_m=F(A)$ is the free group with fixed basis
\[
 A=\{a_1,\dots,a_m\},
\]
and $T_A$ is its unit-edge Cayley tree.  We identify elements of $F_m$ with their freely reduced words over $A^{\pm1}$.  For $u\in F_m$, let $|u|_A$ be its freely reduced length, $[u]$ its conjugacy class, and $\|u\|_A$ the cyclically reduced length of $[u]$, equivalently the translation length of $u$ in $T_A$.  Subscripts are omitted when the basis is clear.  A \emph{positive word} is a word over $A$, and $A^n$ denotes the set of positive words of length $n$.  This convention remains in force throughout the paper unless explicitly overridden.
\end{conv}

\begin{defn}[Root-free and irredundant tuples]\label{def:irredundant}
\begin{enumerate}[(1)]
\item A nontrivial element $g\in F_m$ is \emph{root-free} if it is not a proper power.
\item A tuple $(s_1,\dots,s_t)$ of nontrivial elements of $F_m$ is \emph{root-free} if every entry is root-free.
\item A tuple of nontrivial elements of $F_m$ is \emph{irredundant} if no two distinct entries are conjugate to one another or to one another's inverses.
\end{enumerate}
\end{defn}

For a set $R$ of nontrivial cyclically reduced words in $F_m$, its \emph{symmetrization} $\mathcal R$ consists of all cyclic permutations of elements of $R^{\pm 1}$; and such a set $R$ is \emph{symmetrized} if $R=\mathcal R$.
A nonempty word $v$ is a \emph{piece} of a symmetrized set $\calR$ if it is a common initial segment of two distinct elements of $\calR$.  The set $\calR$ satisfies $C'(\lambda)$ if every piece $v$ beginning $r\in\calR$ satisfies
\[
 |v|<\lambda |r|_A.
\]
Let $\mathbf g=(g_1,\dots,g_q)$ be a tuple of nontrivial elements of $F_m$, and choose cyclically reduced representatives $u_i\in[g_i]$.  We denote by $\calR_A(\mathbf g)$ the symmetrization of the set $\{u_1,\dots, u_q\}$.  These conventions are consistent with the standard small-cancellation terminology of~\cite[Chapter~V]{LS77}.

 For $\phi\in\Out(F_m)$, define $\calR_A(\phi(\mathbf g)):=\calR_A(\Phi(\mathbf g))$ and $\|\phi(g_i)\|_A:=\|\Phi(g_i)\|_A$ using any representative $\Phi\in \Aut(F_m)$ of the outer automorphism $\phi$. These notations are well defined because inner automorphisms preserve conjugacy classes. 

\begin{defn}[$\lambda$-stability]\label{def:lambda-stability}
Let $0<\lambda<1$.
\begin{enumerate}[(1)]
\item A set $\mathcal R\subseteq F_m\setminus\{1\}$ is \emph{$\lambda$-stable} if, for every $\Phi\in\Aut(F_m)$, the symmetrization of cyclically reduced forms of the elements of $\Phi(\mathcal R)$ satisfies $C'(\lambda)$.
\item A tuple is $\lambda$-stable if its underlying set is $\lambda$-stable.
\item A nontrivial element is $\lambda$-stable if its singleton set is $\lambda$-stable.
\end{enumerate}
\end{defn}

Let $v$ be a nontrivial freely reduced word over $A^{\pm1}$ and $w=x_0\cdots x_{\ell-1}$ a nontrivial cyclically reduced word of length $\ell$.  Put
\[
 (v,w)_A
 :=\#\left\{s\in\Z/\ell\Z:
 x_sx_{s+1}\cdots x_{s+|v|_A-1}=v\right\},
\]
where subscripts are read modulo $\ell$.  Thus $(v,w)_A$ is the number of nonsymmetrized cyclic occurrences of $v$ in $w$.  Put
\[
 \langle v,w\rangle_A
 :=(v,w)_A+(v^{-1},w)_A
 =(v^{\pm1},w)_A.
\]
Thus $\langle v,w\rangle_A$ is the corresponding symmetrized cyclic occurrence count; the notation $(v^{\pm1},w)_A$ is shorthand for the displayed sum.  Both quantities depend only on the cyclic word represented by $w$.

For later counting, the number of freely reduced words of length $n\ge 1$ is
\[
 |\FR_n|=2m(2m-1)^{n-1}.
\]
The number of cyclically reduced words is
\begin{equation}\label{eq:number-cyclically-reduced}
 |\CR_n|=(2m-1)^n+m+(m-1)(-1)^n.
\end{equation}
Indeed, $|\CR_n|$ is the trace of the adjacency matrix on $A^{\pm1}$ in which a letter may be followed by every letter except its inverse; that matrix has eigenvalues $2m-1$, $1$ with multiplicity $m$, and $-1$ with multiplicity $m-1$.

A \emph{signed relabeling automorphism}, or simply a \emph{relabeling automorphism}, is an automorphism induced by a permutation of $A^{\pm1}$ commuting with inversion.  The finite group of such automorphisms is denoted $\Rel(A)$ and has order
\[
 |\Rel(A)|=2^m m!.
\]

\subsection{Trees and Outer space}\label{subsec:outer-prelim}

We use Culler--Vogtmann Outer space and its length-function compactification~\cite{CV86,CM87,CL95,Gui00,LL03,FM11}.  Let $\cv_m$ be the space, up to equivariant isometry, of free, minimal, discrete simplicial metric $F_m$-trees with finite quotient, and let $\cvbar_m$ be its length-function closure.  Its nonzero points are nontrivial minimal very small trees: tripod stabilizers are trivial and nondegenerate arc stabilizers are trivial or maximal cyclic; see also~\cite{GL07,Gui08}.  For $T\in\cvbar_m$ and $g\in F_m$, let
\[
 \|g\|_T:=\inf_{x\in T}d_T(x,gx)
\]
be the translation length.  If $T$ is free and simplicial and $g\ne1$, then $g$ has an axis, denoted $A_T(g)$.

We use the standard right action of $\Out(F_m)$ on $\cvbar_m$, denoted
$T\phi$, and characterized by
\[
 \|g\|_{T\phi}=\|\phi(g)\|_T
 \qquad
 (T\in\cvbar_m,\ \phi\in\Out(F_m),\ g\in F_m).
\]
Thus $T_A\phi$ is the pullback of $T_A$ by $\phi$; it may be represented by the same underlying metric tree as $T_A$, with the $F_m$-action precomposed by any automorphism representing $\phi$.  In particular,
\[
 \|g\|_{T_A\phi}=\|\phi(g)\|_A.
\]
\begin{defn}[Lipschitz factor and normalized slice]\label{def:LambdaA}
Let $T\in\cvbar_m\setminus\{0\}$.
\begin{enumerate}[(1)]
\item The \emph{$A$-Lipschitz factor} of $T$ is
\[
 \Lambda_A(T)
 :=\inf\{\Lip(f):f:T_A\to T\text{ is }F_m\text{-equivariant}\}.
\]
\item The corresponding \emph{normalized slice} is
\[
 \calK_A:=\{T\in\cvbar_m:\Lambda_A(T)=1\}.
\]
\item For $\phi\in\Out(F_m)$, put
\[
 \Lambda_A(\phi):=\Lambda_A(T_A\phi).
\]
\end{enumerate}
\end{defn}

The definition of $\Lambda_A(T)$ has a concrete displacement interpretation.  An
$F_m$-equivariant map $f:T_A\to T$ is determined, after straightening the
images of edges to geodesic segments, by the image $x=f(1)$ of the identity
vertex.  Since every edge of $T_A/F_m$ has length one, the straightened map has
Lipschitz constant
\[
 \max_{a\in A}d_T(x,ax).
\]
Consequently,
\begin{equation}\label{eq:Lambda-displacement}
 \Lambda_A(T)=\inf_{x\in T}\max_{a\in A}d_T(x,ax).
\end{equation}

For $\phi\in\Out(F_m)$, the tree $T_A\phi$ lies in $\cv_m$.  Passing to
quotient marked graphs identifies an equivariant map
$T_A\to T_A\phi$ with a change-of-marking map between roses.  In the
terminology of Francaviglia--Martino~\cite[Definition~3.7]{FM11}, such a
piecewise-linear map is \emph{optimal} if at every vertex of its maximally
stretched subgraph there is a legal turn, equivalently, the incident
maximally stretched edge germs are not all mapped into one common direction.
They prove that optimal maps exist and that the stretching factor of every
optimal map realizes the minimal Lipschitz constant
~\cite[Proposition~3.11]{FM11}.  Thus an equivariant optimal map
$f:T_A\to T_A\phi$ may be chosen with
\[
 \Lip(f)=\Lambda_A(\phi).
\]

\begin{prop}[Outer-space facts]\label{prop:outer-space-facts}
The following hold.
\begin{enumerate}[(1)]
\item The projectivized compactification $\mathbb P\cvbar_m$ is compact.
\item The function $\Lambda_A$ is positive, continuous, and homogeneous on $\cvbar_m\setminus\{0\}$, and $\calK_A$ is compact.
\end{enumerate}
\end{prop}

\begin{proof}
Culler--Morgan compactness~\cite[Theorem~4.5]{CM87} gives part~\textup{(1)}; see also~\cite{CL95,Gui00,LL03}.  Homogeneity of \eqref{eq:Lambda-displacement} is immediate, and continuity in the length-function topology is standard~\cite{CM87,FM11}; compare~\cite{GL07,Gui08}.  For positivity, if $\Lambda_A(T)=0$, choose $x_k\in T$ with $\max_{a\in A}d_T(x_k,ax_k)\to0$.  The displacement of each word at $x_k$ is bounded by its word length times this maximum, so every $g\in F_m$ is elliptic.  A finitely generated group acting on an $\mathbb R$-tree with every element elliptic has a global fixed point, contradicting nontrivial minimality.  Thus every projective class has a unique representative with $\Lambda_A(T)=1$, and
\[
 [T]\longmapsto \frac{T}{\Lambda_A(T)}
\]
is a homeomorphism $\mathbb P\cvbar_m\to\calK_A$, so $\calK_A$ is compact.
\end{proof}

There is also a useful finite description of the quantity in
Definition~\ref{def:LambdaA}.  Proposition~3.15 of~\cite{FM11} says that the
maximal stretching between two marked metric graphs is realized by a
candidate loop.  When the source is the standard rose $T_A/F_m$, the
candidates are represented by cyclically reduced words of $A$-length at most
two.  Hence, for $T\in\cv_m$,
\begin{equation}\label{eq:Lambda-short-words}
 \Lambda_A(T)
 =\max\left\{
 \frac{\|u\|_T}{|u|_A}:
 u\text{ cyclically reduced},\ 1\le |u|_A\le2
 \right\}.
\end{equation}
Thus $\Lambda_A(T)$ is determined by the translation lengths $\|u\|_T$ of
freely reduced words $u$ with $1\le |u|_A\le2$.  Since $\cv_m$ is dense in
$\cvbar_m$, the left-hand side of \eqref{eq:Lambda-short-words} is continuous
by Proposition~\ref{prop:outer-space-facts}, and the right-hand side is a
finite maximum of translation-length coordinates, the same formula holds for
every nontrivial $T\in\cvbar_m$.

If $T,T'$ are free, minimal, cocompact simplicial $F_m$-trees and $f:T\to T'$ is an $F_m$-equivariant Lipschitz map, its \emph{bounded-backtracking constant} is
\[
 \BBT(f)=\sup\{d_{T'}(f(z),[f(x),f(y)]):z\in[x,y]\}.
\]
The map $f$ map is a quasi-isometry, hence induces $\partial T\cong\partial T'$.  The following key facts records a particularly useful $BBT(f)$ estimate~\cite[Lemma 3.1]{BFH97}:

\begin{prop}[BBT bound]\label{prop:BBT-bound}
Let $T,T'\in\cv_m$ (where $m\ge 2$) and let $f:T\to T'$ be an $F_m$-equivariant Lipschitz map.
\begin{equation}\label{eq:BBT-standard}
 \BBT(f)\le \Lip(f)\vol(T/F_m),
\end{equation}
where $\vol(T/F_m)$ is the total edge length of the quotient graph.
\end{prop}

In particular, an equivariant $K$-Lipschitz map $T_A\to T'$ has bounded-backtracking constant at most $mK$.

\subsection{Geodesic currents and the intersection form}\label{subsec:currents-prelim}

Let $\partial F_m$ be the Gromov boundary of $F_m$ and put
\[
 \dd F_m=(\partial F_m\times\partial F_m)
 \setminus\{(\xi,\xi):\xi\in\partial F_m\}.
\]
Let $\iota:\dd F_m\to\dd F_m$ be the flip $\iota(\xi,\zeta)=(\zeta,\xi)$.  We identify $\dd F_m$ with the space of oriented bi-infinite geodesics in $T_A$.  A \emph{geodesic current} on $F_m$ is a positive locally finite Borel measure on $\dd F_m$ invariant under the diagonal action of $F_m$ and under $\iota$; see~\cite{KL09,KL10}.  Write $\Curr(F_m)$ for the space of geodesic currents and
\[
 \PCurr(F_m):=(\Curr(F_m)\setminus\{0\})/\R_{>0}.
\]
We equip $\Curr(F_m)$ with the weak-$\ast$ topology.

For $\nu\in\Curr(F_m)$, its \emph{support} is
\[
 \supp(\nu)
 =\{z\in\dd F_m:\nu(U)>0
       \text{ for every open neighborhood }U\ni z\}.
\]
It is a closed $F_m$- and $\iota$-invariant subset of $\dd F_m$.  The boundary action of $\Aut(F_m)$ induces an $\Out(F_m)$-action on currents, since inner automorphisms act trivially on $F_m$-invariant measures, and
\[
 \supp(\phi\nu)=\phi\bigl(\supp(\nu)\bigr)
 \qquad(\phi\in\Out(F_m)).
\]
Our action conventions give $\phi\eta_g=\eta_{\phi(g)}$ and
\begin{equation}\label{eq:intersection-equivariance}
 \langle T_A,\phi\eta\rangle
 =\langle T_A\phi,\eta\rangle
 \qquad(\eta\in\Curr(F_m),\ \phi\in\Out(F_m)).
\end{equation}
Positive rescaling does not change support.

For $1\ne g\in F_m$, let $g^{+\infty},g^{-\infty}\in\partial F_m$ be its attracting and repelling fixed points and put
\[
 \ell_g=(g^{-\infty},g^{+\infty})\in\dd F_m.
\]
The elements of $F_m\cdot\ell_g\cup F_m\cdot\iota(\ell_g)$ are the \emph{periodic leaves associated to $[g]$}.  If $\eta_g$ is the counting current of $[g]$, then
\[
 \supp(\eta_g)=F_m\cdot\ell_g\cup F_m\cdot\iota(\ell_g).
\]
For a nontrivial freely reduced word $v$, let $\operatorname{Cyl}_A(v)\subseteq\dd F_m$ be the cylinder of oriented geodesics containing the oriented segment $[1,v]\subseteq T_A$, and put
\[
 \langle v,\nu\rangle_A
 :=\nu\bigl(\operatorname{Cyl}_A(v)\bigr).
\]
For a nontrivial cyclically reduced word $w$, the cyclic-occurrence notation from Section~\ref{subsec:free-small-prelim} satisfies
\[
 \langle v,\eta_w\rangle_A=\langle v,w\rangle_A.
\]
The cylinder coordinates determine the weak-$\ast$ topology on $\Curr(F_m)$: for every $\nu\in\Curr(F_m)$, neighborhoods specified by finitely many inequalities
\[
 |\langle v,\eta\rangle_A-\langle v,\nu\rangle_A|<\epsilon_v
\]
with nontrivial freely reduced words $v$ and constants $\epsilon_v>0$ form a neighborhood basis at $\nu$.

The geometric intersection form
\[
 \langle\cdot,\cdot\rangle:
 \cvbar_m\times\Curr(F_m)\longrightarrow\R_{\ge0}
\]
is continuous and satisfies
\[
 \langle T,\eta_g\rangle=\|g\|_T.
\]
A nonzero current $\nu$ is \emph{filling} in $\Curr(F_m)$ if $\langle T,\nu\rangle>0$ for every nontrivial $T\in\cvbar_m$.  For boundary trees, dual laminations, and currents, see~\cite{CHL08a,CHL08b,CHL08c}.  We use the following standard facts~\cite{CM87,KL09,KL10,KL15}.

\begin{prop}[Current facts]\label{prop:current-facts}
The following hold.
\begin{enumerate}[(1)]
\item The intersection form is continuous, homogeneous in the tree coordinate, and linear in the current coordinate.
\item For every $\phi\in \Out(F_m), T\in \cvbar_m, \nu\in \Curr(F_m)$ we have
\[
\langle T\phi, \nu\rangle=\langle T, \phi\nu\rangle.
\]
\item For $T\in\cvbar_m\setminus\{0\}$ and $\nu\in\Curr(F_m)$,
\[
 \langle T,\nu\rangle=0
 \quad\Longleftrightarrow\quad
 \supp(\nu)\subseteq L^2(T),
\]
where $L^2(T)\subseteq\dd F_m$ is the dual algebraic lamination of $T$; this is the zero-intersection criterion of~\cite[Theorem~1.1]{KL10}.
\end{enumerate}
\end{prop}

\section{Signed periodic overlaps and automorphic pieces}\label{sec:signed-overlap}

We identify signed periodic overlaps with intersections of translates of relator axes.

\begin{lem}[Uniqueness of symmetrized representatives]\label{lem:symmetrized-uniqueness}
Let $\mathbf g=(g_1,\dots,g_q)$ be a root-free irredundant tuple of nontrivial elements of $F_m$, and choose cyclically reduced representatives $u_i$ of $[g_i]$.  If a cyclic shift of $u_i^\epsilon$ equals a cyclic shift of $u_j^\delta$, where $\epsilon,\delta\in\{\pm1\}$, then $i=j$, $\epsilon=\delta$, and the starting positions agree modulo $|u_i|$.  Thus every member of $\calR_A(\mathbf g)$ has a unique origin triple consisting of a component, an orientation, and a cyclic position.
\end{lem}

\begin{proof}
If $i\ne j$, equality makes $g_i$ conjugate to $g_j^{\pm1}$, contrary to irredundancy.  If $i=j$ and $\epsilon=\delta$, distinct starting positions give a nontrivial cyclic period, contrary to root-freeness.  Finally, no nontrivial free-group element is conjugate to its inverse.  If $hgh^{-1}=g^{-1}$, then $h^2$ centralizes $g$.  If $h^2=1$, torsion-freeness gives $h=1$ and $g=g^{-1}$; if $h^2\ne1$, then $g,h$ lie in the cyclic centralizer of $h^2$ and commute.  Both are impossible.
\end{proof}

\begin{defn}[Admissibility and automorphic piece ratio]\label{def:admissible}
A tuple $\mathbf g=(g_1,\dots,g_q)$ of nontrivial elements of $F_m$ is \emph{admissible} if it is root-free and irredundant.

For an admissible tuple define
\[
 \DeltaA(\mathbf g)
 =\sup_{\phi\in\Out(F_m)}
 \sup\left\{
 \frac{|v|}{|r|_A}:
 \begin{array}{l}
 r\in\calR_A(\phi(\mathbf g)),\\
 v\text{ is a piece of }\calR_A(\phi(\mathbf g))\\
 \text{and an initial segment of }r
 \end{array}
 \right\},
\]
where the inner supremum is zero when there are no pieces.  If $\mathbf g$ is not admissible, put $\DeltaA(\mathbf g)=\infty$.
\end{defn}

Admissibility is invariant under automorphisms, and
\begin{equation}\label{eq:Delta-implies-stability}
 \DeltaA(\mathbf g)<\lambda
 \quad\Longrightarrow\quad
 \mathbf g\text{ is admissible and }\lambda\text{-stable}.
\end{equation}

Let
\[
 \mathbf w=(w_1,\dots,w_q),
 \qquad
 w_i=x_{i,0}\cdots x_{i,\ell_i-1}
\]
be a tuple of nontrivial cyclically reduced words.  Indices in the $i$-th component are read modulo $\ell_i$.  For $s\in\Z/\ell_i\Z$, $\epsilon\in\{+,-\}$, and $r\in\Z$, put
\[
 \omega_{i,+,s}(r)=x_{i,s+r},
 \qquad
 \omega_{i,-,s}(r)=x_{i,s-r-1}^{-1}.
\]
Thus $\omega_{i,+,s}$ is a cyclic shift of the bi-infinite periodic word $w_i^\infty$, while $\omega_{i,-,s}$ is a cyclic shift of $(w_i^{-1})^\infty$.

\begin{defn}[Signed periodic-overlap length]\label{def:rho-signed}
For a tuple $\mathbf w$ of nontrivial cyclically reduced words, with signed periodic words $\omega_{i,\epsilon,s}$ as above, define $\rhoSigned(\mathbf w)$ as follows.  It is $0$ if no distinct signed origins exist, $\infty$ if two distinct signed periodic words agree identically, and otherwise the largest integer $R\ge0$ for which distinct signed origins satisfy
\[
 \omega_{i,\epsilon,s}[0,R)
 =\omega_{j,\delta,t}[0,R),
\]
where $\omega[0,R)$ denotes the word $\omega(0)\omega(1)\cdots\omega(R-1)$, with the empty-word convention when $R=0$.
\end{defn}

\begin{lem}[Infinite signed overlap]\label{lem:rho-signed-admissible}
Let $\mathbf w$ be a tuple of nontrivial cyclically reduced words.  Then
\[
 \rhoSigned(\mathbf w)<\infty
 \quad\Longleftrightarrow\quad
 \mathbf w\text{ is admissible}.
\]
\end{lem}

\begin{proof}
Identical same-orientation shifts of one component give a nontrivial cyclic period, hence a proper power.  Identical shifts from distinct components make their cyclic words, up to conjugacy and inversion, powers of a common primitive cyclic word.  An exponent of absolute value greater than one violates root-freeness; if both exponents have absolute value one, irredundancy fails.  Identical opposite-orientation shifts of one component would make the element conjugate to its inverse, impossible by Lemma~\ref{lem:symmetrized-uniqueness}.  Thus infinite overlap implies nonadmissibility.  Conversely, a proper power yields identical same-orientation shifts, and conjugate-up-to-inversion components yield identical signed periodic words after suitable shifts.
\end{proof}

\begin{lem}[Pieces and signed periodic overlaps]\label{lem:pieces-rho-signed}
Let $\mathbf w=(w_1,\dots,w_q)$ be an admissible tuple of nontrivial cyclically reduced words.  Let
\[
 P_A(\mathbf w)=\max\{|v|:v\text{ is a piece of }\calR_A(\mathbf w)\},
\]
where $P_A(\mathbf w)=0$ if $\calR_A(\mathbf w)$ has no pieces.  Then:
\begin{enumerate}[(1)]
\item
\[
 P_A(\mathbf w)\le \rhoSigned(\mathbf w).
\]
\item If $|w_i|_A=n$ for every $1\le i\le q$, then
\[
 P_A(\mathbf w)=\rhoSigned(\mathbf w).
\]
\end{enumerate}
\end{lem}

\begin{proof}
Let $v$ be a piece of $\calR_A(\mathbf w)$.  Thus $v$ is a common initial segment of two distinct members of $\calR_A(\mathbf w)$.  By Lemma~\ref{lem:symmetrized-uniqueness}, these two members have distinct signed origins.  Their corresponding signed periodic words therefore have a common initial segment of length $|v|$, and hence $|v|\le\rhoSigned(\mathbf w)$.  This proves part~\textup{(1)}.

Suppose now that $|w_i|_A=n$ for every $i$.  By admissibility and Lemma~\ref{lem:rho-signed-admissible}, $\rhoSigned(\mathbf w)<\infty$.  Moreover,
\[
 \rhoSigned(\mathbf w)<n.
\]
Indeed, if two distinct signed origins had a common periodic prefix of length at least $n$, then the corresponding signed periodic words, both $n$-periodic, would agree identically, contrary to admissibility.  Choose distinct signed origins realizing $\rhoSigned(\mathbf w)$.  Since this overlap has length less than $n$, it is also a common initial segment of the corresponding two length-$n$ members of $\calR_A(\mathbf w)$, which are distinct by Lemma~\ref{lem:symmetrized-uniqueness}.  Thus it is a piece, so
\[
 P_A(\mathbf w)\ge\rhoSigned(\mathbf w).
\]
Together with part~\textup{(1)}, this proves part~\textup{(2)}.
\end{proof}

\begin{ex}[Strict inequality in Lemma~\ref{lem:pieces-rho-signed}]\label{ex:strict-piece-rho}
Let $F_2=F(a,b)$ and
\[
 \mathbf w=(ab,aba).
\]
The tuple $\mathbf w$ is admissible.  The signed periodic words
\[
 (ab)^\infty=ababab\cdots,
 \qquad
 (aba)^\infty=abaaba\cdots
\]
have common initial segment $aba$ of length $3$, and a direct check of the finitely many signed origins shows that no longer common prefix occurs.  Thus
\[
 \rhoSigned(\mathbf w)=3.
\]
On the other hand, $ab$ is a piece of $\calR_A(\mathbf w)$, while no piece has length greater than $2$; hence
\[
 P_A(\mathbf w)=2<3=\rhoSigned(\mathbf w).
\]
The strict inequality occurs because the periodic overlap continues past the end of the shorter relator $ab$.
\end{ex}

Let $T$ be a free simplicial $F_m$-tree.  For $g\ne1$, let $A_T(g)$ be its axis.  For an admissible tuple $\mathbf g=(g_1,\dots,g_q)$ put
\[
 \Ov_T(\mathbf g)
 =\sup\diam(\ell\cap\ell'),
\]
where the supremum ranges over distinct lines
\[
 \ell=hA_T(g_i),
 \qquad
 \ell'=h'A_T(g_j),
 \qquad h,h'\in F_m.
\]
The diameter of the empty set is zero.

\begin{lem}[Pieces and intersections of axes]\label{lem:pieces-axes}
Let $\mathbf g$ be an admissible tuple of nontrivial cyclically reduced words.  Every piece $v$ of $\calR_A(\mathbf g)$ determines two distinct translates of component axes whose intersection contains a segment of length $|v|$.  Consequently,
\[
 |v|\le\Ov_{T_A}(\mathbf g).
\]
\end{lem}

\begin{proof}
Let distinct $r,r'\in\calR_A(\mathbf g)$ have common initial segment $v$.  By Lemma~\ref{lem:symmetrized-uniqueness}, they have distinct component-orientation-position triples, each determining an oriented fundamental segment on a component-axis translate.  Translating the common initial vertex to the identity makes the two copies of $v$ a common segment of these translates.

The underlying lines in $T_A$ are distinct.  Otherwise their preferred orientations agree, since opposite orientations cannot read a nonempty common word.  Their common line $F_m$-stabilizer is generated by the primitive root of either component, forcing distinct components to be conjugate up to inversion or, for one component, root-freeness to make the triples equal.
\end{proof}

\begin{lem}[Signed periodic factors are axis overlaps]\label{lem:signed-axis-overlap}
Let $\mathbf w$ be an admissible tuple of nontrivial cyclically reduced words.  Then
\[
 \Ov_{T_A}(\mathbf w)=\rhoSigned(\mathbf w).
\]
\end{lem}

\begin{proof}
For the $i$-th component and $s\in\Z/\ell_i\Z$, put
\[
 p_{i,s}=x_{i,0}\cdots x_{i,s-1},
 \qquad p_{i,0}=1.
\]
The line $p_{i,s}^{-1}A_{T_A}(w_i)$ passes through the identity.  In its positive translation direction its forward label is the one-sided restriction of $\omega_{i,+,s}$; in the opposite direction its forward label is the one-sided restriction of $\omega_{i,-,s}$.

A common signed periodic factor of length $R$ therefore gives an axis intersection of length $R$.  For $R>0$ the lines are distinct: otherwise their agreeing orientations and first common edge force the complete signed periodic words to agree, contrary to admissibility.  Hence
\[
 \Ov_{T_A}(\mathbf w)\ge\rhoSigned(\mathbf w).
\]

Conversely, orient two distinct component-axis translates compatibly on a nondegenerate intersection and translate its initial vertex to the identity.  The oriented lines determine distinct signed origins, and the intersection reads a common prefix of their periodic words.  An unbounded intersection would give a common ray and, by periodicity, identical signed periodic words, contrary to admissibility.  Thus its length is at most $\rhoSigned(\mathbf w)$.
\end{proof}

We next control line intersections under equivariant Lipschitz maps.

\begin{lem}[Distortion of line intersections]\label{lem:line-overlap}
Let $T,T'$ be free, minimal, cocompact simplicial $F_m$-trees, and let $f:T\to T'$ be an $F_m$-equivariant Lipschitz map.  Put $K=\Lip(f)$ and $B=\BBT(f)$.  Let $\ell_1,\ell_2\subseteq T$ be distinct geodesic lines, and let $\ell_1',\ell_2'\subseteq T'$ be the lines whose endpoint pairs are the images of those of $\ell_1,\ell_2$ under the boundary homeomorphism induced by $f$.  If $\diam(\ell_1\cap\ell_2)<\infty$, then
\[
 \diam(\ell_1'\cap\ell_2')
 \le2K\diam(\ell_1\cap\ell_2)+4B.
\]
\end{lem}

\begin{proof}
For distinct $\xi,\eta\in\partial T$ and $x\in T$, let $z$ be the projection of $x$ to the geodesic line with endpoints $\xi,\eta$.  Thus
\[
 d_T(x,z)=(\xi\mid\eta)_x.
\]
Choose points $p_n,q_n$ on the two rays of this line based at $z$ and tending to $\xi,\eta$, respectively.  Since $z\in[p_n,q_n]$, bounded backtracking gives
\[
 d_{T'}\bigl(f(z),[f(p_n),f(q_n)]\bigr)\le B.
\]
The segments $[f(p_n),f(q_n)]$ converge on compact sets to the line with endpoints $\xi',\eta'$.  Hence that line comes within distance $B$ of $f(z)$, while
\[
 d_{T'}(f(x),f(z))\le Kd_T(x,z).
\]
Since a Gromov product in a tree is the distance from the basepoint to the line joining the two boundary points, we obtain
\begin{equation}\label{eq:ray-product}
 (\xi'\mid\eta')_{f(x)}\le K(\xi\mid\eta)_x+B.
\end{equation}

Put $r=\diam(\ell_1\cap\ell_2)$.  If the lines meet, choose $x\in\ell_1\cap\ell_2$; otherwise choose on $\ell_1$ the endpoint of their bridge.  For every endpoint $\xi$ of $\ell_1$ and $\eta$ of $\ell_2$,
\[
 (\xi\mid\eta)_x\le r.
\]
Also $d_{T'}(f(x),\ell_1')\le B$.  Since $r<\infty$, both pairs of lines have disjoint endpoint sets.  Let $J=\ell_1'\cap\ell_2'$ and project $f(x)$ to $x'\in\ell_1'$.  If $|J|=p>0$, suitable endpoints $\xi'$ of $\ell_1'$ and $\eta'$ of $\ell_2'$ satisfy $(\xi'\mid\eta')_{x'}\ge p/2$.  Changing basepoint to $f(x)$ and using \eqref{eq:ray-product} gives
\[
 p/2-B\le(\xi'\mid\eta')_{f(x)}\le Kr+B.
\]
Thus $p\le2Kr+4B$.
\end{proof}

\begin{prop}[Uniform deterministic piece bound]\label{prop:deterministic-piece}
Let $\mathbf w$ be an admissible tuple of nontrivial cyclically reduced words.  For every $\phi\in\Out(F_m)$ and every piece $v$ of $\calR_A(\phi(\mathbf w))$,
\begin{equation}\label{eq:det-piece}
 |v|\le2\Lambda_A(\phi)\bigl(\rhoSigned(\mathbf w)+2m\bigr).
\end{equation}
The estimate applies to self-pieces and to pieces shared by different relators.
\end{prop}

\begin{proof}
Choose an automorphism $\Phi$ representing $\phi$.  By Lemma~\ref{lem:pieces-axes}, a piece $v$ of $\calR_A(\phi(\mathbf w))$ gives distinct translates of axes of the $\Phi(w_i)$ in $T_A$.  In the pullback realization of $T_A\phi$, these are the endpoint images, under an optimal map $T_A\to T_A\phi$, of the corresponding source-axis translates.  Their source intersection is at most $\rhoSigned(\mathbf w)$ by Lemma~\ref{lem:signed-axis-overlap}.  Applying Lemma~\ref{lem:line-overlap} with
\[
 K=\Lambda_A(\phi),
 \qquad
 B\le m\Lambda_A(\phi)
\]
from \eqref{eq:BBT-standard} in Proposition~\ref{prop:BBT-bound} yields
\[
|v|\le \Ov_{T_A}(\phi(\mathbf w)) \le 2K\rhoSigned(\mathbf w)+4B\le 2\Lambda_A(\phi)\rhoSigned(\mathbf w)+4m\Lambda_A(\phi)=2\Lambda_A(\phi)\bigl(\rhoSigned(\mathbf w)+2m\bigr).
\]
\end{proof}

\begin{rem}
The estimate is deterministic; positivity in the earlier positive-word argument only replaced signed periodic overlaps by unsigned ones.
\end{rem}

\section{Filling currents and abstract stability criteria}\label{sec:abstract-stability}

\subsection{Finite-frequency stability}

Recall from Definition~\ref{def:LambdaA} that
\[
 \calK_A=\{T\in\cvbar_m:\Lambda_A(T)=1\},
\]
and that $\calK_A$ is compact by Proposition~\ref{prop:outer-space-facts}.

Put
\[
 \|\eta\|_A:=\langle T_A,\eta\rangle
 \qquad(\eta\in\Curr(F_m))
\]
and
\[
 \mathcal S_A=\{\eta\in\Curr(F_m):\|\eta\|_A=1\}.
\]
Here and below, $\mathcal S_A$ carries the subspace topology inherited from the weak-$\ast$ topology on $\Curr(F_m)$.
Since
\[
 \|\eta\|_A=\sum_{a\in A}\langle a,\eta\rangle_A>0
\]
for every nonzero current, each projective current has a unique representative in $\mathcal S_A$.

\begin{lem}[Uniform convergence on the compact slice]\label{lem:uniform-slice}
Let $\mu_j,\mu\in\Curr(F_m)$, and suppose that $\mu_j\to\mu$.  Then:
\begin{enumerate}[(1)]
\item
\[
 \sup_{T\in\calK_A}
 |\langle T,\mu_j\rangle-\langle T,\mu\rangle|
 \longrightarrow0
 \qquad\text{as }j\to\infty.
\]
\item The function
\[
 \mathfrak m_A(\eta):=\min_{T\in\calK_A}\langle T,\eta\rangle
\]
is continuous on $\Curr(F_m)$.
\end{enumerate}
\end{lem}

\begin{proof}
If (1) fails then there are $\epsilon>0$, a subsequence $j_k\to\infty$, and trees $T_k\in\calK_A$ such that
\[
 |\langle T_k,\mu_{j_k}\rangle-\langle T_k,\mu\rangle|
 \ge\epsilon
\]
for every $k\ge 1$.  After passing to a further subsequence, compactness of $\calK_A$ gives $\displaystyle \lim_{k\to\infty}T_k= T$ for some $T\in\calK_A$.  Since $\mu_{j_k}\to\mu$, joint continuity of the intersection form gives
\[
 \lim_{k\to\infty} |\langle T_k,\mu_{j_k}\rangle-\langle T_k,\mu\rangle|=|\langle T,\mu\rangle-\langle T,\mu\rangle|=0
\]
contradicting the preceding lower bound.  This proves part~\textup{(1)}.

For part~\textup{(2)}, the estimate
\[
 |\mathfrak m_A(\mu_j)-\mathfrak m_A(\mu)|
 \le\sup_{T\in\calK_A}
 |\langle T,\mu_j\rangle-\langle T,\mu\rangle|
\]
and part~\textup{(1)} give the continuity of $\mathfrak m_A$.
\end{proof}

\begin{lem}[Uniform filling neighborhood]\label{lem:uniform-filling-neighborhood}
Let $\nu\in\mathcal S_A$ be filling.  There are a neighborhood $U$ of $\nu$ in $\mathcal S_A$ and $c>0$ such that
\begin{enumerate}
\item We have
\begin{equation}\label{eq:uniform-filling-neighborhood}
\langle T,\eta\rangle\ge c\Lambda_A(T)
\end{equation}
for every $\eta\in U$ and every nontrivial $T\in\cvbar_m$. 

\item For every $\phi\in \Out(F_m)$ we have
\begin{equation}\label{eq:uniform-auto-distortion}
\langle T_A\phi,\eta\rangle=\langle T_A,\phi\eta\rangle\ge c\Lambda_A(\phi)
\end{equation}
\end{enumerate}
\end{lem}

\begin{proof}
Since $\nu$ is filling and $\calK_A$ is compact, $\mathfrak m_A(\nu)>0$.  Choose $0<c<\mathfrak m_A(\nu)$ and a neighborhood $U$ on which $\mathfrak m_A(\eta)>c$.  For $T\ne0$, the tree $T/\Lambda_A(T)$ lies in $\calK_A$, and homogeneity gives \eqref{eq:uniform-filling-neighborhood}, so that part (1) holds.

For part (2), for any $\phi\in \Out(F_m)$, by part (1) we have
\[
\langle T_A,\phi\eta\rangle=\langle T_A\phi,\eta\rangle\ge c\Lambda_A(T_A\phi)=c\Lambda_A(\phi),
\]
as required.
\end{proof}

\begin{prop}[Finite-frequency deterministic estimate]\label{prop:finite-frequency-estimate}
Fix $q\ge 1$ and filling currents
\[
 \nu_1,\dots,\nu_q\in\mathcal S_A.
\]
There exist an integer $M\ge 1$ and constants $0<\epsilon<1$ and $0<c\le1$ with the following property.  Let
\[
 \mathbf W=(W_1,\dots,W_q)
\]
be a tuple of nontrivial cyclically reduced words, and put
\[
 \ell_i=\|W_i\|_A=|W_i|_A,
 \qquad
 \ell_*=\min_{1\le i\le q}\ell_i.
\]
Suppose that, for every $1\le i\le q$ and every nontrivial freely reduced word $v$ with $|v|_A\le M$,
\begin{equation}\label{eq:finite-frequency-hypothesis}
 \left|
 \frac{\langle v,W_i\rangle_A}{\ell_i}
 -\langle v,\nu_i\rangle_A
 \right|<\epsilon.
\end{equation}
Then, for every $1\le i\le q$ and every $\phi\in\Out(F_m)$,
\begin{equation}\label{eq:finite-frequency-denominator}
 \|\phi(W_i)\|_A\ge c\ell_i\Lambda_A(\phi).
\end{equation}
If, in addition, $\mathbf W$ is admissible, then
\begin{equation}\label{eq:finite-frequency-Delta}
 \Delta_A(\mathbf W)
 \le
 \frac{2\bigl(\rhoSigned(\mathbf W)+2m\bigr)}{c\ell_*}.
\end{equation}
\end{prop}

\begin{proof}
For each $i$, choose a neighborhood $U_i\subseteq\mathcal S_A$ of $\nu_i$ and $c_i>0$ as in Lemma~\ref{lem:uniform-filling-neighborhood}.  Using the cylinder-coordinate neighborhood basis from Subsection~\ref{subsec:currents-prelim}, choose a finite nonempty set $\mathcal V_i$ of nontrivial freely reduced words and positive numbers $\epsilon_{i,v}$, $v\in\mathcal V_i$, such that
\[
 \left\{\eta\in\mathcal S_A:
 |\langle v,\eta\rangle_A-\langle v,\nu_i\rangle_A|
 <\epsilon_{i,v}
 \text{ for every }v\in\mathcal V_i\right\}
 \subseteq U_i.
\]
Choose $M\ge 1$ at least the length of every word in the finitely many sets $\mathcal V_i$, and choose $0<\epsilon<1$ no larger than every $\epsilon_{i,v}$.  Then, for every $i$, the inequalities
\[
 |\langle v,\eta\rangle_A-\langle v,\nu_i\rangle_A|<\epsilon
 \qquad(1\le |v|_A\le M)
\]
for $\eta\in\mathcal S_A$ imply $\eta\in U_i$.  Put
\[
 c=\min\{1,c_1,\dots,c_q\}.
\]

For $1\le i\le q$, let
\[
 \widehat\eta_i=\frac{1}{\ell_i}\eta_{W_i}.
\]
Since $W_i$ is nontrivial and cyclically reduced, one has $\widehat\eta_i\in\mathcal S_A$, and
\[
 \langle v,\widehat\eta_i\rangle_A
 =\frac{\langle v,W_i\rangle_A}{\ell_i}.
\]
Thus \eqref{eq:finite-frequency-hypothesis} gives $\widehat\eta_i\in U_i$.  For every $\phi\in\Out(F_m)$, equivariance and part~\textup{(2)} of Lemma~\ref{lem:uniform-filling-neighborhood} now give
\[
 \frac{1}{\ell_i}\|\phi(W_i)\|_A
 =\langle T_A\phi,\widehat\eta_i\rangle
 \ge c_i\Lambda_A(\phi)
 \ge c\Lambda_A(\phi),
\]
which proves \eqref{eq:finite-frequency-denominator}.

Suppose now that $\mathbf W$ is admissible.  If a piece $u$ of $\calR_A(\phi(\mathbf W))$ begins a relator arising from component $i$, Proposition~\ref{prop:deterministic-piece} and \eqref{eq:finite-frequency-denominator} give
\[
 \frac{|u|}{\|\phi(W_i)\|_A}
 \le
 \frac{2\Lambda_A(\phi)\bigl(\rhoSigned(\mathbf W)+2m\bigr)}
 {c\ell_i\Lambda_A(\phi)}
 \le\frac{2\bigl(\rhoSigned(\mathbf W)+2m\bigr)}{c\ell_*}.
\]
Taking the supremum over all pieces and all $\phi$ proves \eqref{eq:finite-frequency-Delta}.
\end{proof}

\phantomsection\label{proof:finite-frequency-stability}
\begin{mainthmrestated}{thm:finite-frequency-stability}{Finite-frequency stability criterion}
Fix $q\ge 1$, filling currents
\[
 \nu_1,\dots,\nu_q\in\Curr(F_m)
 \qquad\text{with}\qquad
 \langle T_A,\nu_i\rangle=1
\]
for $1\le i\le q$, and $0<\lambda<1$.  Then there exist constants $0<\epsilon<1$ and $0<\lambda'<\lambda$, and integers $M,n_0\ge 1$, such that the following holds.  Let
\[
 \mathbf W=(W_1,\dots,W_q)
\]
be a root-free irredundant tuple of nontrivial cyclically reduced words, not necessarily of equal lengths, and put
\[
 \ell_i=\|W_i\|_A=|W_i|_A
 \qquad(1\le i\le q).
\]
Suppose that $\ell_i\ge n_0$ for every $i$, that the symmetrization of $\mathbf W$ satisfies $C'(\lambda')$, and that, for every $1\le i\le q$ and every nontrivial freely reduced word $v$ with $|v|_A\le M$,
\[
 \left|
 \frac{\langle v,W_i\rangle_A}{\ell_i}
 -\langle v,\nu_i\rangle_A
 \right|<\epsilon.
\]
Then
\[
 \Delta_A(\mathbf W)<\lambda.
\]
In particular, $\mathbf W$ is $\lambda$-stable.
\end{mainthmrestated}

\begin{proof}
Apply Proposition~\ref{prop:finite-frequency-estimate} to $\nu_1,\dots,\nu_q$, obtaining $M$, $\epsilon$, and $0<c\le1$.  Choose
\[
 0<\lambda'<\frac{c\lambda}{4}
\]
and then choose $n_0\ge M$ so that
\[
 \frac{4m}{cn_0}<\frac{\lambda}{2}.
\]
Since $c\le1$, one has $\lambda'<\lambda$.

Let $\mathbf W$ satisfy the hypotheses.  Proposition~\ref{prop:finite-frequency-estimate} gives
\begin{equation}\label{eq:finite-frequency-variable-denominator}
 \|\phi(W_i)\|_A\ge c\ell_i\Lambda_A(\phi)
 \qquad(1\le i\le q,\ \phi\in\Out(F_m)).
\end{equation}
Fix $\phi\in\Out(F_m)$, and let $u$ be a piece of $\calR_A(\phi(\mathbf W))$ beginning a relator arising from component $i$.  Choose an automorphism $\Phi$ representing $\phi$.  As in the proof of Proposition~\ref{prop:deterministic-piece}, Lemma~\ref{lem:pieces-axes}, applied to cyclically reduced representatives of the $\Phi(W_k)$, associates to $u$ two distinct target-axis translates whose intersection contains a segment of length $|u|$.  Let $f:T_A\to T_A\phi$ be an optimal equivariant map.  In the pullback realization of $T_A\phi$, these target lines are the endpoint images under $f$ of distinct source-axis translates $L_1,L_2\subseteq T_A$, where $L_1$ belongs to component $i$ and $L_2$ belongs to a component $j$.  Put
\[
 d=\diam(L_1\cap L_2).
\]
This number is finite by Lemma~\ref{lem:signed-axis-overlap}.

We claim that
\begin{equation}\label{eq:Cprime-componentwise-source-overlap}
 d<\lambda'\ell_i.
\end{equation}
There is nothing to prove if $d=0$.  If $d>0$, orient the source lines compatibly along their intersection and translate an initial vertex of the intersection to the identity.  As in the proof of Lemma~\ref{lem:signed-axis-overlap}, the intersection label is then a common prefix of signed periodic words arising from $W_i$ and $W_j$, with distinct signed origins.  If
\[
 d\ge\min\{\ell_i,\ell_j\},
\]
and $\ell_i\ne\ell_j$, the full shorter cyclic word is a piece of $\calR_A(\mathbf W)$, contradicting $C'(\lambda')$ because $\lambda'<1$.  If $\ell_i=\ell_j$, the two corresponding symmetrized relators agree in all $\ell_i$ positions, contradicting Lemma~\ref{lem:symmetrized-uniqueness}.  Hence $d<\min\{\ell_i,\ell_j\}$.  The corresponding members of $\calR_A(\mathbf W)$ are distinct, and the full intersection label is therefore a piece beginning a cyclic conjugate of $W_i^{\pm1}$.  The $C'(\lambda')$ hypothesis gives \eqref{eq:Cprime-componentwise-source-overlap}.

Apply Lemma~\ref{lem:line-overlap} to the optimal map $T_A\to T_A\phi$.  Using \eqref{eq:BBT-standard} and \eqref{eq:Cprime-componentwise-source-overlap}, we obtain
\[
 |u|
 \le2\Lambda_A(\phi)d+4m\Lambda_A(\phi)
 <2\Lambda_A(\phi)(\lambda'\ell_i+2m).
\]
Together with \eqref{eq:finite-frequency-variable-denominator}, this gives
\[
 \frac{|u|}{\|\phi(W_i)\|_A}
 <\frac{2\lambda'}{c}+\frac{4m}{c\ell_i}
 \le\frac{2\lambda'}{c}+\frac{4m}{cn_0}
 <\lambda.
\]
Taking the supremum over all pieces and all $\phi$ proves $\Delta_A(\mathbf W)<\lambda$.  The implication \eqref{eq:Delta-implies-stability} completes the proof.
\end{proof}

\begin{rem}\label{rem:finite-frequency-effectivity}
Theorem~\ref{thm:finite-frequency-stability} is a finite-data sufficient condition, although the proof does not assert that suitable $M$, $\epsilon$, $\lambda'$, and $n_0$ can be computed effectively from an arbitrary description of the filling currents $\nu_i$.
\end{rem}

\subsection{Probabilistic consequences}

\begin{defn}[Exponential concentration]\label{def:exponential-concentration}
Let $\mu_n$ be random currents taking values in $\mathcal S_A$, and let $\nu\in\mathcal S_A$.  We say that $\mu_n$ \emph{concentrates exponentially} at $\nu$ if, for every neighborhood $U$ of $\nu$ in $\mathcal S_A$, there are constants $C_U,c_U>0$ such that
\[
 \Pr(\mu_n\notin U)\le C_Ue^{-c_Un}
\]
for every $n\ge 1$.
\end{defn}

\begin{rem}[Why fillingness is essential]\label{rem:filling-essential}
The filling hypothesis in Proposition~\ref{prop:finite-frequency-estimate} and Theorem~\ref{thm:finite-frequency-stability} is essential.  Let $m\ge3$, let $W_n$ be a uniform positive word of length $n$ in $F(a_2,\dots,a_m)$, and put
\[
 U_n=a_1W_n.
\]
Let $\nu_0\in\mathcal S_A$ be the current characterized by
\[
 \langle v,\nu_0\rangle_A
 =\begin{cases}
 (m-1)^{-|v|_A},
 &v\text{ or }v^{-1}\text{ is a positive word over }\{a_2,\dots,a_m\},\\
 0,&\text{otherwise}.
 \end{cases}
\]
The normalized counting currents $(n+1)^{-1}\eta_{U_n}$ concentrate exponentially at $\nu_0$ in the sense of Definition~\ref{def:exponential-concentration}: the Bernoulli frequencies in $W_n$ have exponentially decaying deviation probabilities on every fixed finite family of cylinder coordinates, while the single letter $a_1$ affects each fixed normalized coordinate by only $O(1/n)$.  The current $\nu_0$ is not filling.  Indeed, it has zero intersection with the Bass--Serre tree of
\[
 F_m=\langle a_1\rangle * F(a_2,\dots,a_m),
\]
because $F(a_2,\dots,a_m)$ is elliptic there.

The denominator conclusion fails accordingly.  Each $U_n$ is primitive, since $a_1\mapsto a_1W_n$, $a_i\mapsto a_i$ for $i\ge 2$, defines an automorphism.  Hence some $\phi_n\in\Out(F_m)$ satisfies
\[
 \|\phi_n(U_n)\|_A=1.
\]
Because $\Lambda_A(\phi_n)\ge 1$, no $c>0$ can satisfy
\[
 \|\phi(U_n)\|_A\ge c(n+1)\Lambda_A(\phi)
\]
for all $n$ and $\phi$.  Thus this nonstable family lacks precisely the filling-current denominator.
\end{rem}

The next two results derive the probabilistic stability estimates from the deterministic finite-frequency estimate.

\begin{prop}[Filling-current denominator estimate]\label{prop:abstract-denominator}
Fix $q\ge 1$ and filling currents $\nu_1,\dots,\nu_q\in\mathcal S_A$.  For every $n\ge 1$ and every $1\le i\le q$, let $W_{i,n}$ be a random cyclically reduced word of length $n$, and suppose that
\[
 \widehat\eta_{i,n}:=\frac{1}{n}\eta_{W_{i,n}}
\]
concentrates exponentially at $\nu_i$ in the sense of Definition~\ref{def:exponential-concentration} for every $i$.  Then there are constants
\[
 0<c_*\le1,\qquad C_0,c_0>0
\]
such that
\[
 \Pr\left(
 \|\phi(W_{i,n})\|_A\ge c_*n\Lambda_A(\phi)
 \text{ for all }i\text{ and }\phi\in\Out(F_m)
 \right)
 \ge 1-C_0e^{-c_0n}.
\]
\end{prop}

\begin{proof}
Choose $M$, $\epsilon$, and $c$ from Proposition~\ref{prop:finite-frequency-estimate}.  For $1\le i\le q$, let $V_i\subseteq\mathcal S_A$ be the open neighborhood of $\nu_i$ defined by
\[
 |\langle v,\eta\rangle_A-\langle v,\nu_i\rangle_A|<\epsilon
 \qquad(1\le |v|_A\le M).
\]
Exponential concentration and a union bound show that all $\widehat\eta_{i,n}$ lie in their corresponding $V_i$ outside an event of probability at most $C_0e^{-c_0n}$.  On the complementary event, Proposition~\ref{prop:finite-frequency-estimate} gives \eqref{eq:finite-frequency-denominator} for every $i$ and $\phi$.  Take $c_*=c$.
\end{proof}

\begin{thm}[Abstract master estimate]\label{thm:abstract-master}
Fix $q\ge 1$.  For every $n\ge 1$, let $\mathbf W_n=(W_{1,n},\dots,W_{q,n})$ be a random tuple of cyclically reduced words of common length $n$.  Suppose:
\begin{enumerate}
\item[\textup{(A1)}] There are filling currents $\nu_1,\dots,\nu_q\in\mathcal S_A$ such that, for every $1\le i\le q$, the random currents
\[
 \widehat\eta_{i,n}:=\frac{1}{n}\eta_{W_{i,n}}
\]
concentrate exponentially at $\nu_i$ in the sense of Definition~\ref{def:exponential-concentration};
\item[\textup{(A2)}] There are constants $D\ge0$, $C_1>0$, and $0<\theta<1$ such that, for every $n\ge 1$ and every integer $R$ with $1\le R\le n/4$,
\begin{equation}\label{eq:abstract-overlap-tail}
 \Pr\bigl(\rhoSigned(\mathbf W_n)\ge R\bigr)
 \le C_1n^D\theta^R.
\end{equation}
\end{enumerate}
Then there are constants $c_*>0$ and $C_0,c_0>0$ such that the following hold.
\begin{enumerate}[(1)]
\item For every $n\ge 1$ and every integer $R$ with $1\le R\le n/4$,
\begin{equation}\label{eq:abstract-master}
 \Pr\left(
 \DeltaA(\mathbf W_n)>
 \frac{2(R+2m)}{c_*n}
 \right)
 \le C_0e^{-c_0n}+C_1n^D\theta^R.
\end{equation}
\item One has
\[
 \DeltaA(\mathbf W_n)
 =O_{\Pr}\left(\frac{\log n}{n}\right),
\]
and for every fixed $\epsilon>0$ there are constants $C_\epsilon,c_\epsilon>0$ such that
\[
 \Pr\bigl(\DeltaA(\mathbf W_n)>\epsilon\bigr)
 \le C_\epsilon e^{-c_\epsilon n}
\]
for every $n\ge 1$.
\item For every fixed $0<\lambda<1$, the tuple $\mathbf W_n$ is admissible and $\lambda$-stable with exponentially high probability.
\end{enumerate}
\end{thm}

\begin{proof}
Choose $M$, $\epsilon_0$, and $0<c\le1$ from Proposition~\ref{prop:finite-frequency-estimate}, and put $c_*=c$.  For each $n$, let $E_n$ be the event that
\[
 \left|
 \langle v,\widehat\eta_{i,n}\rangle_A
 -\langle v,\nu_i\rangle_A
 \right|<\epsilon_0
\]
for every $1\le i\le q$ and every nontrivial freely reduced word $v$ with $|v|_A\le M$.  Assumption~\textup{(A1)} and a finite union bound give constants $C_0,c_0>0$ such that
\[
 \Pr(E_n^c)\le C_0e^{-c_0n}
\]
for every $n\ge 1$.

On $E_n\cap\{\rhoSigned(\mathbf W_n)<R\}$, the tuple is admissible by Lemma~\ref{lem:rho-signed-admissible}, and Proposition~\ref{prop:finite-frequency-estimate} gives
\[
 \Delta_A(\mathbf W_n)
 \le\frac{2\bigl(\rhoSigned(\mathbf W_n)+2m\bigr)}{c_*n}
 <\frac{2(R+2m)}{c_*n}.
\]
Combining the two exceptional-event bounds proves part~\textup{(1)}.

For the first assertion of part~\textup{(2)}, take $R=\lceil K\log n\rceil$ for all sufficiently large $n$, with $K$ large enough that $n^D\theta^R\to0$; for such $n$ one also has $R\le n/4$.  This gives
\[
 \Delta_A(\mathbf W_n)=O_{\Pr}\left(\frac{\log n}{n}\right).
\]
For fixed $\epsilon>0$, take
\[
 R_n=\left\lfloor\frac{c_*\min\{\epsilon,1\}}{4}n\right\rfloor.
\]
For all sufficiently large $n$, one has $1\le R_n\le n/4$ and
\[
 \frac{2(R_n+2m)}{c_*n}<\epsilon.
\]
Since $R_n$ grows linearly in $n$, the term $n^D\theta^{R_n}$ is exponentially small.  Part~\textup{(1)} therefore gives the fixed-scale exponential estimate; enlarging $C_\epsilon$ covers the finitely many remaining values of $n$.  Part~\textup{(3)} follows from \eqref{eq:Delta-implies-stability} and the fixed-scale estimate with $\epsilon=\lambda/2$.
\end{proof}

\section{Uniformly random reduced words}\label{sec:random-reduced}

Put
\[
 \Sigma=A^{\pm1},\qquad d=|\Sigma|=2m.
\]
We first prove the signed-overlap estimate for uniform cyclically reduced words.

\subsection{Signed periodic-overlap tails}

\begin{lem}[Conditional coordinates in a random cyclic word]\label{lem:conditional-coordinates}
Let $n\ge 2$, and let
\[
 X=X_0X_1\cdots X_{n-1}
\]
be uniform in $\CR_n$, and let $I\subseteq\Z/n\Z$ contain no two cyclically adjacent positions.  Conditional on all coordinates $X_j$ with $j\notin I$, the variables $X_i$, $i\in I$, are independent.  Each $X_i$ is uniform on
\[
 \Sigma\setminus\{X_{i-1}^{-1},X_{i+1}^{-1}\},
\]
a set of cardinality at least $d-2=2m-2$.
\end{lem}

\begin{proof}
Condition on an outside labeling of positive probability.  Since $I$ contains no adjacent positions, each remaining reducedness constraint involves one variable and its fixed neighbors.  Uniformity of admissible completions and factorization of their number give the asserted independence and conditional laws.
\end{proof}

\begin{prop}[Signed periodic-overlap tail]\label{prop:signed-overlap-tail}
Fix $q\ge 1$.  Let $\mathbf C_n=(C_{1,n},\dots,C_{q,n})$ be an independent uniformly random tuple in $\CR_n^q$.  There are constants $C_1>0$ and $0<\theta<1$, depending only on $m,q$, such that, for every $n\ge 1$ and every integer $R$ with $1\le R\le n/4$,
\begin{equation}\label{eq:signed-overlap-tail}
 \Pr\bigl(\rhoSigned(\mathbf C_n)\ge R\bigr)
 \le C_1n^2\theta^R.
\end{equation}
One may take $\theta=(2m-2)^{-1/18}$ after enlarging $C_1$.
\end{prop}

\begin{proof}
Fix two distinct signed origins.  Equality of their length-$R$ prefixes imposes constraints of the form
\[
 X_u=X_v
 \quad\text{or}\quad
 X_u=X_v^{-1}.
\]

If the origins lie in different components, choose at least $R/3$ pairwise nonadjacent positions in the relevant interval of the second component, discarding at most one endpoint at the cyclic cut.  Conditional on the first component and all other coordinates of the second, Lemma~\ref{lem:conditional-coordinates} makes these constraints independent, each with probability at most $(d-2)^{-1}$.  Their agreement probability is therefore at most $(d-2)^{-R/3}$, up to a fixed factor for small $R$.

Suppose the origins lie in one component.  With equal signs, the constraints join $s+r$ to $t+r$ for $0\le r<R$, so every vertex has degree at most two.  Distinct $r$ give distinct ordered edges.  Repetition as an unordered edge with reversed endpoints would imply $2(t-s)=0$ modulo $n$, hence $t-s=n/2$ and $r-r'\equiv n/2$ modulo $n$, impossible for $0\le r,r'<R\le n/4$.  Thus all $R$ edges are distinct.

If the signs are opposite, the constraint edges have the form
\[
 \{s+r,t-r-1\},\qquad0\le r<R.
\]
A loop makes the event empty.  Otherwise each unordered edge occurs at most twice, since an edge and one endpoint determine $r$ up to interchanging the endpoints.  Each cyclic position occurs at most once in each endpoint role, so the graph has at least $R/2$ distinct edges and maximum degree two.

Thus the constraint graph has at least $R/2$ edges and maximum degree two.  It has a matching of size at least $R/6$; choosing one endpoint from each matching edge and then an independent set in the cyclic adjacency graph gives a set $I$ of size at least $R/18$.  Orient the corresponding matching constraints toward $I$.  Their other endpoints are distinct and outside $I$, so after conditioning off $I$, each constraint prescribes a different coordinate.  Lemma~\ref{lem:conditional-coordinates} bounds their joint probability by
\[
 (d-2)^{-|I|}\le(d-2)^{-R/18}.
\]

There are $2qn$ signed origins and fewer than $4q^2n^2$ ordered pairs.  A union bound proves the result.
\end{proof}

\begin{cor}[Variable-length signed-overlap tail]\label{cor:signed-overlap-variable}
Fix $q\ge 1$ and positive integers $\ell_1,\dots,\ell_q$.  Let $C_1,\dots,C_q$ be independent uniformly random cyclically reduced words of lengths $\ell_1,\dots,\ell_q$, respectively.  Put
\[
 \ell_*=\min_i\ell_i,
 \qquad
 L=\sum_{i=1}^q\ell_i.
\]
For every integer $R$ with $1\le R\le\ell_*/4$,
\[
 \Pr\bigl(\rhoSigned(C_1,\dots,C_q)\ge R\bigr)
 \le 4L^2(2m-2)^{-R/18}.
\]
\end{cor}

\begin{proof}
The preceding proof uses only $R\le\ell_i/4$ for components containing an origin, so it applies unchanged.  There are $2L$ signed origins and fewer than $4L^2$ ordered pairs.
\end{proof}

\begin{cor}[Exponential admissibility]\label{cor:cr-admissibility}
Fix $q\ge 1$, and let $\mathbf C_n$ be the independent uniformly random tuple in $\CR_n^q$ from Proposition~\ref{prop:signed-overlap-tail}.  There are constants $C,c>0$, depending only on $m,q$, such that, for every $n\ge 1$, the tuple $\mathbf C_n$ is admissible with probability at least $1-Ce^{-cn}$.
\end{cor}

\begin{proof}
For $n\ge4$, nonadmissibility is equivalent to $\rhoSigned(\mathbf C_n)=\infty$ by Lemma~\ref{lem:rho-signed-admissible}, and Proposition~\ref{prop:signed-overlap-tail} applies with $R=\lfloor n/4\rfloor$.  Enlarging the prefactor covers $n<4$.
\end{proof}

\subsection{The uniform nonbacktracking current}

Define the \emph{uniform current} $\nu_A\in\Curr(F_m)$ by
\begin{equation}\label{eq:uniform-current-coordinates}
 \langle v,\nu_A\rangle_A
 =\frac{1}{m(2m-1)^{|v|-1}}
\end{equation}
for every nontrivial freely reduced word $v$.  These coordinates satisfy $\|\nu_A\|_A=1$.  Every reduced cylinder has positive measure, so
\[
 \supp(\nu_A)=\dd F_m.
\]

Since $\nu_A$ has full support, Corollary~1.6 of \cite{KL09} implies:

\begin{lem}\label{lem:uniform-current-filling}
The current $\nu_A\in \Curr(F_m)$ is filling.
\end{lem}

\begin{lem}[Exponential current concentration]\label{lem:cr-current-concentration}
For every $n\ge 1$, let $C_n$ be uniform in $\CR_n$.  Then $\frac{1}{n}\eta_{C_n}$ concentrates exponentially at $\nu_A$ in the sense of Definition~\ref{def:exponential-concentration}.  Moreover, for every fixed $0<\alpha\le1$ and every neighborhood $U$ of $\nu_A$ in $\mathcal S_A$, there are constants $C_U,c_U>0$ such that
\[
 \Pr\left(\frac{1}{\ell}\eta_{C_\ell}\notin U\right)
 \le C_Ue^{-c_UN}
\]
whenever $N\ge 1$ and $\alpha N\le\ell\le N$.
\end{lem}

\begin{proof}
It suffices to control finitely many cylinder coordinates.  Let $(Y_k)$ be the stationary nonbacktracking chain on $\Sigma$, with uniform initial distribution and transitions
\[
 \Pr(Y_{k+1}=y\mid Y_k=x)
 =\begin{cases}
 (d-1)^{-1},&y\ne x^{-1},\\
 0,&y=x^{-1}.
 \end{cases}
\]
Then $Y_0\cdots Y_{n-1}$ is uniform in $\FR_n$, and conditioning on
\[
 E_n=\{Y_{n-1}\ne Y_0^{-1}\}
\]
gives the uniform law on $\CR_n$.  For $n\ge 2$, every entry of every positive transition-matrix power is at most $(d-1)^{-1}$, since each later row is a convex combination of first-power rows.  Hence
\[
 \Pr(E_n)\ge\frac{d-2}{d-1}.
\]
(The case $n=1$ is immediate.)

For a reduced word $v$ of length $r$, its non-wrapping occurrence count is an additive functional of the finite-state chain of reduced $r$-blocks.  Thus the same applies, up to the bounded wrapping error, to the cyclic count $(v,C_n)_A$.  The large-deviation estimate of~\cite[Section~5]{KSS06} gives exponential concentration of the normalized count at its stationary mean
\[
 \frac{1}{d(d-1)^{r-1}}.
\]
Conditioning on $E_n$ changes the probability by at most the fixed factor $(d-1)/(d-2)$.  For $v$, the cyclic and non-wrapping occurrence counts differ by at most $r-1$, and the same is true for $v^{-1}$; hence the corresponding symmetrized counts differ by at most $2(r-1)$.  Applying the large-deviation estimate to both $v$ and $v^{-1}$ gives exponential concentration of
\[
 \frac{1}{n}\langle v,\eta_{C_n}\rangle_A
 =\frac{1}{n}\langle v,C_n\rangle_A
\]
at
\[
 \frac{2}{d(d-1)^{r-1}}
 =\frac{1}{m(2m-1)^{r-1}},
\]
which is \eqref{eq:uniform-current-coordinates}.  Since neighborhoods specified by finitely many cylinder-coordinate inequalities form a basis for the current topology, a finite union bound proves exponential concentration.  The finitely many values of $n$ smaller than one of the block lengths occurring in such a neighborhood are absorbed by increasing the prefactor.  For the final assertion, apply this estimate at length $\ell$ and use $e^{-c\ell}\le e^{-c\alpha N}$ whenever $\alpha N\le\ell\le N$.
\end{proof}

\subsection{Uniform $\lambda$-stability for random reduced words}

\phantomsection\label{proof:mainthm-stability}
\begin{mainthmrestated}{mainthm:stability}{Uniform $\lambda$-stability}
Fix $m\ge 2$, $q\ge 1$, and $0<\lambda<1$.  Let
\[
 \mathbf C_n=(C_{1,n},\dots,C_{q,n})
\]
be an independent uniformly random $q$-tuple in $\CR_n^q$.  Then the following hold.
\begin{enumerate}[(1)]
\item There are constants $C,c>0$ such that, for every $n\ge 1$, with probability at least $1-Ce^{-cn}$ the tuple $\mathbf C_n$ is root-free, irredundant, and $\lambda$-stable.

\item More precisely,
\[
 \Delta_A(\mathbf C_n)=O_{\Pr}\left(\frac{\log n}{n}\right),
\]
and for every fixed $\epsilon>0$ there are $C_\epsilon,c_\epsilon>0$ such that
\[
 \Pr\bigl(\Delta_A(\mathbf C_n)>\epsilon\bigr)
 \le C_\epsilon e^{-c_\epsilon n}
\]
for every $n\ge 1$.

\item If $W_{1,n},\dots,W_{q,n}$ are independent uniformly random words in $\FR_n$ and $\widehat W_{i,n}$ is the cyclically reduced form of $W_{i,n}$, then the conclusions of parts~\textup{(1)}--\textup{(2)} hold for
\[
 \widehat{\mathbf W}_n=(\widehat W_{1,n},\dots,\widehat W_{q,n});
\]
in particular,
\[
 \Delta_A(\widehat{\mathbf W}_n)
 =O_{\Pr}\left(\frac{\log n}{n}\right),
\]
and the corresponding fixed-scale tails are exponential.
\end{enumerate}
\end{mainthmrestated}

\begin{proof}
Proposition~\ref{prop:signed-overlap-tail} and Lemmas~\ref{lem:uniform-current-filling}--\ref{lem:cr-current-concentration} verify the hypotheses of Theorem~\ref{thm:abstract-master}, proving the assertions for $\mathbf C_n$.

It remains to treat freely reduced words.  Every $W\in\FR_n$ has a unique decomposition
\begin{equation}\label{eq:cyclic-reduction-decomposition}
 W=uCu^{-1},
\end{equation}
where $C$ is cyclically reduced.  Put $K=|u|$ and $L=|C|=n-2K$.

Conditional on $L=\ell$, the word $C$ is uniform in $\CR_\ell$.  For $k=(n-\ell)/2\ge 1$ and fixed $C=c_1\cdots c_\ell$, a reduced $u=x_1\cdots x_k$ gives \eqref{eq:cyclic-reduction-decomposition} exactly when
\[
 x_k\ne c_1^{-1},
 \qquad
 x_k\ne c_\ell.
\]
The two forbidden letters are distinct, so the number of such $u$ is
\[
 (d-2)(d-1)^{k-1},
\]
independent of $C$.  The case $k=0$ is immediate.

Moreover, for $t\ge 1$,
\begin{equation}\label{eq:cyclic-cancellation-tail}
 \Pr(K\ge t)
 \le\frac{d-1}{d-2}(d-1)^{-t}.
\end{equation}
Indeed, use $|\CR_{n-2k}|\le d(d-1)^{n-2k-1}$ and the preceding count, divide by $|\FR_n|=d(d-1)^{n-1}$, and extend the sum over $k\ge t$ to an infinite geometric series.

For the tuple, \eqref{eq:cyclic-cancellation-tail} puts all cyclic-reduction lengths in $[n/2,n]$ outside an exponentially small event.  Conditional on these lengths, the cyclic reductions are independent and uniform.  Choose $M$, $\epsilon$, and $c$ from Proposition~\ref{prop:finite-frequency-estimate} in the case $q=1$ and $\nu_1=\nu_A$.  Uniformly for $|C_i|\in[n/2,n]$, Lemma~\ref{lem:cr-current-concentration} shows that the corresponding finite-frequency inequalities hold simultaneously outside an exponentially small event.  Applying the denominator conclusion of Proposition~\ref{prop:finite-frequency-estimate} separately to each singleton $(C_i)$ therefore gives
\[
 \|\phi(C_i)\|_A\ge c\,|C_i|\Lambda_A(\phi)
 \ge(c/2)n\Lambda_A(\phi)
\]
for every $i$ and $\phi$ on that event.  Corollary~\ref{cor:signed-overlap-variable} supplies the overlap estimate uniformly for $R\le n/8$.  The proof of Theorem~\ref{thm:abstract-master}, with $n/8$ in place of $n/4$, then gives
\[
 \Delta_A(\widehat{\mathbf W}_n)
 =O_{\Pr}\left(\frac{\log n}{n}\right)
\]
with exponential tails at every fixed positive scale; admissibility and $\lambda$-stability follow.
\end{proof}

\begin{correstated}{cor:sphere-ball-stability}{Freely reduced sphere and ball models}
Fix $m\ge 2$, $q\ge 1$, and $0<\lambda<1$.
\begin{enumerate}[(1)]
\item Let
\[
 \mathbf W_n=(W_{1,n},\dots,W_{q,n})
\]
be an independent uniformly random $q$-tuple in $\FR_n^q$.  Then there are constants $C,c>0$ such that, for every $n\ge 1$, with probability at least $1-Ce^{-cn}$ the tuple $\mathbf W_n$ is root-free, irredundant, and $\lambda$-stable.  Moreover,
\[
 \Delta_A(\mathbf W_n)=O_{\Pr}\left(\frac{\log n}{n}\right),
\]
and for every fixed $\epsilon>0$ there are $C_\epsilon,c_\epsilon>0$ such that
\[
 \Pr\bigl(\Delta_A(\mathbf W_n)>\epsilon\bigr)
 \le C_\epsilon e^{-c_\epsilon n}
\]
for every $n\ge 1$.

\item Let
\[
 B_A(n)=\{g\in F_m:|g|_A\le n\},
\]
and let $\mathbf U_n=(U_{1,n},\dots,U_{q,n})$ be uniformly distributed on $B_A(n)^q$ (equivalently, its components are independent uniformly random freely reduced words of length at most $n$).  Then the conclusions of part~\textup{(1)} hold with $\mathbf U_n$ in place of $\mathbf W_n$.
\end{enumerate}
\end{correstated}

\begin{proof}
For part~\textup{(1)}, each $W_{i,n}$ is conjugate to $\widehat W_{i,n}$, and all asserted properties depend only on the component conjugacy classes.  Apply Theorem~\ref{mainthm:stability}(3).

For part~\textup{(2)}, put $N_i=|U_{i,n}|_A$.  Since the number of freely reduced words of length $k$ grows like $(2m-1)^k$, there are constants $C_1,c_1>0$ such that
\[
 \Pr(N_i<3n/4)\le C_1e^{-c_1n}.
\]
Conditional on $N_i$, the reduced word representing $U_{i,n}$ is uniform in $\FR_{N_i}$.  Write
\[
 U_{i,n}=u_iC_i u_i^{-1},
\]
where $C_i$ is cyclically reduced and $K_i=|u_i|_A$.  The cyclic-cancellation estimate~\eqref{eq:cyclic-cancellation-tail}, uniformly in $N_i$, gives
\[
 \Pr(K_i>n/8\mid N_i)\le C_2e^{-c_2n}
\]
for suitable constants.  Hence, outside an exponentially small event,
\[
 n/2\le \ell_i:=|C_i|_A\le n
 \qquad(1\le i\le q).
\]
Conditional on the $\ell_i$, the $C_i$ are independent and uniform in $\CR_{\ell_i}$.  Uniformly for $\ell_i\in[n/2,n]$, Lemma~\ref{lem:cr-current-concentration} and the componentwise denominator conclusion of Proposition~\ref{prop:finite-frequency-estimate} give
\[
 \|\phi(C_i)\|_A\ge c\ell_i\Lambda_A(\phi)
 \ge(c/2)n\Lambda_A(\phi)
\]
outside an exponentially small event, for a constant $c>0$ independent of the lengths.  Corollary~\ref{cor:signed-overlap-variable} provides the overlap estimate.  Repeating the argument from the proof of Theorem~\ref{thm:abstract-master} gives
\[
 \Delta_A(C_1,\dots,C_q)
 =O_{\Pr}\left(\frac{\log n}{n}\right)
\]
and exponential fixed-scale tails, hence admissibility and $\lambda$-stability with exponentially high probability.  Conjugacy invariance transfers these conclusions to $\mathbf U_n$.
\end{proof}

\section{Positive Bernoulli words}\label{sec:positive-corollary}

The abstract criterion also yields the earlier positive-word theorem without positive-specific rigidity machinery.

For a strictly positive probability vector
\[
 \mathbf p=(p_1,\dots,p_m),
 \qquad p_s>0,
 \qquad \sum p_s=1,
\]
let $\nu_{\mathbf p}$ be the characteristic current of the Bernoulli chain on $A$, equivalently the current characterized by
\[
 \langle v,\nu_{\mathbf p}\rangle_A
 =\begin{cases}
 p_{i_1}\cdots p_{i_r},
 &v^\epsilon=a_{i_1}\cdots a_{i_r}
   \text{ for some }\epsilon\in\{\pm1\},\\
 0,&v\text{ contains both positive and negative letters}.
 \end{cases}
\]
Then $\|\nu_{\mathbf p}\|_A=1$, and
\[
 \supp(\nu_{\mathbf p})=\mathcal L_A^+\cup\mathcal L_A^-,
\]
where $\mathcal L_A^\pm$ are the positive and negative oriented line languages.

\begin{lem}[Positive Bernoulli currents are filling]\label{lem:positive-current-filling}
Let $\mathbf p=(p_1,\dots,p_m)$ be a strictly positive probability vector on $A$, and let $\nu_{\mathbf p}$ be the current defined above.  Then $\nu_{\mathbf p}$ is filling in $\Curr(F_m)$.
\end{lem}

\begin{proof}
The support contains the periodic leaves of every $a_i$ and $a_ia_j$, $i<j$.  If $\langle T,\nu_{\mathbf p}\rangle=0$, the zero-intersection criterion gives
\[
 \|a_i\|_T=0,
 \qquad
 \|a_ia_j\|_T=0.
\]
The fixed subtrees of the basis elements therefore intersect pairwise, and the Helly property gives a global fixed point, contradicting minimality of a nontrivial tree.
\end{proof}

\begin{lem}[Bernoulli current concentration]\label{lem:positive-current-concentration}
Let $\mathbf p=(p_1,\dots,p_m)$ be a strictly positive probability vector on $A$.  For every $n\ge 1$, let $W_n=X_1\cdots X_n$ be a positive Bernoulli word with letter distribution $\mathbf p$.  Then $\frac{1}{n}\eta_{W_n}$ concentrates exponentially at $\nu_{\mathbf p}$ in the sense of Definition~\ref{def:exponential-concentration}.
\end{lem}

\begin{proof}
For $n\ge r$, the normalized cyclic occurrence count $\frac{1}{n}(v,W_n)_A$ of a positive word $v$ of length $r$ has the corresponding product expectation.  Changing one letter affects at most $r$ cyclic windows, so McDiarmid's inequality~\cite{McD89} gives exponential concentration.  Negative coordinates follow by flip invariance, mixed-sign coordinates vanish, and a finite union bound over the coordinates defining a neighborhood proves exponential concentration in the sense of Definition~\ref{def:exponential-concentration}.  The finitely many smaller values of $n$ are absorbed by increasing the prefactor.
\end{proof}

\begin{cor}[Positive Bernoulli stability]\label{cor:positive-stability}
Fix $q\ge 1$, $0<\lambda<1$, and strictly positive probability vectors $\mathbf p_1,\dots,\mathbf p_q$ on $A$.  For every $n\ge 1$, let $W_{1,n},\dots,W_{q,n}$ be independent positive Bernoulli words of length $n$ with their indicated distributions, and put
\[
 \mathbf W_n=(W_{1,n},\dots,W_{q,n}).
\]
Then the following hold.
\begin{enumerate}[(1)]
\item There are constants $C,c>0$ such that, for every $n\ge 1$, the tuple $\mathbf W_n$ is root-free, irredundant, and $\lambda$-stable with probability at least $1-Ce^{-cn}$.
\item One has
\[
 \Delta_A(\mathbf W_n)=O_{\Pr}\left(\frac{\log n}{n}\right),
\]
and for every fixed $\epsilon>0$ there are constants $C_\epsilon,c_\epsilon>0$ such that
\[
 \Pr\bigl(\Delta_A(\mathbf W_n)>\epsilon\bigr)
 \le C_\epsilon e^{-c_\epsilon n}
\]
for every $n\ge 1$.
\end{enumerate}
\end{cor}

\begin{proof}
Put
\[
 \beta=\max_{1\le i,j\le q}
 \sum_{s=1}^m p_{i,s}p_{j,s}<1.
\]
Two prescribed positive factors from distinct components agree with probability at most $\beta^R$.  Within one component, for $R\le n/4$ the $R$ equality constraints form a maximum-degree-two graph with $R$ distinct edges; a matching of size at least $R/3$ gives probability at most $\beta^{R/3}$.  Opposite signs cannot share a nonempty prefix, while inversion and reversal reduce two negative factors to the positive case.  A union bound over same-sign origin pairs gives
\[
 \Pr\bigl(\rhoSigned(W_{1,n},\dots,W_{q,n})\ge R\bigr)
 \le q^2n^2\beta^{R/3}.
\]
Lemmas~\ref{lem:positive-current-filling} and~\ref{lem:positive-current-concentration} verify the current hypothesis of Theorem~\ref{thm:abstract-master}, which now gives the result.
\end{proof}

We conclude this section with an almost-sure consequence for general finitely supported group random walks.  Unlike the preceding sphere and Bernoulli models, no exponential estimate for stability itself is asserted here.

\begin{cor}[Stability along finitely supported random walks]\label{cor:finitely-supported-random-walk}
Let $\mu$ be a finitely supported probability measure on $F_m$ such that the semigroup generated by $\supp(\mu)$ is $F_m$.  Let
\[
 W_n=X_1\cdots X_n
\]
be the associated random walk, and let $C_n$ be a cyclically reduced representative of the conjugacy class $[W_n]$.  Then, for almost every trajectory,
\[
 \Delta_A(C_n)\longrightarrow0
 \qquad(n\to\infty).
\]
In particular, for every fixed $0<\lambda<1$, the word $C_n$ is root-free and $\lambda$-stable for all sufficiently large $n$, almost surely.
\end{cor}

\begin{proof}
Consider the standard action of $F_m$ on its Cayley tree $T_A$.  The hypothesis on $\supp(\mu)$ makes this action nonelementary for the random walk, and $\mu$ is admissible in the sense of~\cite[Definition~2.4]{MS19}; its support has bounded image since it is finite.  By Maher--Tiozzo~\cite[Theorem~1.4]{MT18}, there are constants $L>0$, $K>0$, and $0<c<1$ such that
\[
 \Pr\bigl(\|W_n\|_A\le Ln\bigr)\le Kc^n.
\]
Since $\|W_n\|_A=|C_n|_A$, Borel--Cantelli gives
\begin{equation}\label{eq:rw-linear-cyclic-length}
 |C_n|_A\ge Ln
\end{equation}
eventually almost surely.

By~\cite[Theorem~D(1)]{Kap26}, there is a filling current $\nu$ to which the random walk is adapted; thus, for almost every trajectory,
\[
 [\eta_{W_n}]=[\eta_{C_n}]\longrightarrow[\nu]
 \qquad\text{in }\PCurr(F_m).
\]
Normalizing by $T_A$-length therefore gives
\begin{equation}\label{eq:rw-current-convergence}
 \frac{\eta_{C_n}}{|C_n|_A}
 \longrightarrow
 \bar\nu:=\frac{\nu}{\langle T_A,\nu\rangle}
 \qquad\text{in }\mathcal S_A
\end{equation}
almost surely, where $\bar\nu$ is filling.

We next use the matching estimates of Maher--Sisto.  Let $\gamma_n=[1,W_n]$ in $T_A$.  Proposition~3.2.2 of~\cite{MS19} says that a self-match of any fixed positive linear size has probability tending to zero.  We use its standard quantitative form in the bounded-support case: tracing its proof through the exponential matching estimate~\cite[Proposition~3.2.1]{MS19} (see also the proof of~\cite[Lemma~5.26]{CM15}) gives, for each fixed positive linear matching threshold, summable failure probabilities in $n$.  Thus Borel--Cantelli applies to these matching events.

For completeness, we indicate why this controls the overlap used here.  Write the freely reduced form of $W_n$ as $U_nC_nU_n^{-1}$, choosing $C_n$ to be the resulting cyclic reduction.  The middle copy of $C_n$ is a fundamental segment of the axis of $W_n$ and is contained in $\gamma_n$.  By Definition~\ref{def:rho-signed}, an inequality
\[
 \rhoSigned(C_n)\ge R,
 \qquad 0<R\le |C_n|_A/4,
\]
including the case $\rhoSigned(C_n)=\infty$, gives two distinct signed periodic origins agreeing for length $R$.  Translating the corresponding axis segments by powers of $W_n$ into a fundamental segment and trimming at cyclic cut points produces two disjoint subsegments of the middle copy of $C_n$, of length at least $aR$, which are translates of one another, possibly with opposite orientations, for an absolute constant $a>0$.  Hence $\gamma_n$ has an $(aR,0)$-match in the terminology of~\cite{MS19}.

Put $B=\max\{|s|_A:s\in\supp(\mu)\}$.  Then $|\gamma_n|\le Bn$, while~\eqref{eq:rw-linear-cyclic-length} gives $|C_n|_A\ge Ln$ eventually almost surely.  Therefore, for every fixed $0<\epsilon<1/4$, the event
\[
 \rhoSigned(C_n)\ge\epsilon |C_n|_A
\]
eventually implies that $\gamma_n$ has a self-match of length at least $(a\epsilon L/B)|\gamma_n|$.  The summable matching estimate and Borel--Cantelli, applied to rational $\epsilon>0$, now give
\begin{equation}\label{eq:rw-sublinear-overlap}
 \frac{\rhoSigned(C_n)}{|C_n|_A}\longrightarrow0
\end{equation}
almost surely.  In particular, $\rhoSigned(C_n)<\infty$ eventually, so Lemma~\ref{lem:rho-signed-admissible} implies that $C_n$ is root-free eventually.

Fix a trajectory for which~\eqref{eq:rw-linear-cyclic-length}--\eqref{eq:rw-sublinear-overlap} hold.  Apply Proposition~\ref{prop:finite-frequency-estimate} with $q=1$ and the filling current $\bar\nu$.  By~\eqref{eq:rw-current-convergence}, its finite-frequency hypothesis holds for all sufficiently large $n$.  Hence, for some $c_0>0$ independent of $n$,
\[
 \Delta_A(C_n)
 \le
 \frac{2\bigl(\rhoSigned(C_n)+2m\bigr)}{c_0|C_n|_A}
 \longrightarrow0.
\]
The final assertion follows from the definition of $\Delta_A$ and $\lambda$-stability.
\end{proof}

\section{Generic Nielsen and Whitehead input}\label{sec:generic-input}

\subsection{Whitehead's algorithm and Whitehead automorphisms}
 We first briefly recall, mainly following~\cite[Proposition~1.2 and Definition~4.2]{KSS06}, the form of Whitehead's algorithm for automorphic equivalence of conjugacy classes that will be used below. See also \cite{LS77} for a detailed treatment of the topic.   Put $\Sigma=A^{\pm1}$.  A \emph{Whitehead automorphism} is an automorphism $\Theta\in\Aut(F_m)$ of one of two types.  A Whitehead automorphism of the \emph{first kind}, or a \emph{relabeling automorphism}, permutes $\Sigma$ and commutes with inversion; these are precisely the elements of $\Rel(A)$.  A Whitehead automorphism of the \emph{second kind} has a multiplier $a\in\Sigma$ such that, for every $x\in\Sigma$,
\[
 \Theta(x)\in\{x,xa,a^{-1}x,a^{-1}xa\};
\]
necessarily $\Theta(a)=a$.  For fixed $m\ge 1$ there are only finitely many Whitehead automorphisms, and they generate $\Aut(F_m)$.

Given two cyclically reduced words $u,v$ in $F_m$, Whitehead's algorithm decides whether there exists $\Phi\in \Aut(F_m)$ such that $\Phi([u])=[v]$. 
For a cyclically reduced word $u$, call $[u]$ \emph{minimal} in its automorphic orbit if $\|\Phi(u)\|_A\ge ||u||_A$ for every $\Phi\in\Aut(F_m)$.  Whitehead's length-reduction theorem says that if $[u]$ is not minimal, then some Whitehead automorphism $\Theta$ strictly decreases its cyclic length, $||\Theta(u)||_A<||u||_A$.  Thus one cyclically reduces and repeatedly applies length-decreasing Whitehead automorphisms until a minimal representative is reached.  Given two conjugacy classes, one first minimizes both; unequal minimal lengths rule out automorphic equivalence.  If the minimal lengths agree, Whitehead's peak-reduction theorem says that the two classes are automorphically equivalent exactly when their minimal representatives can be joined by a finite sequence of Whitehead automorphisms along which cyclic length remains constant.  Since there are only finitely many cyclic words of a fixed length, this procedure requires only a finite search.  In the terminology of~\cite{KSS06}, minimization is the ``easy part'' of Whitehead's algorithm (which runs in $O(|u|_A^2)$ time), while the length-preserving search is the ``hard part'' requiring a priori exponential time in $\max\{|u|_A,|v|_A\}$). The search form records the Whitehead moves and hence produces an explicit automorphism when the classes are equivalent.  Strict minimality, recalled below, makes this second stage especially simple.

\subsection{Nielsen uniqueness and polynomial recognition}

We first record a small algorithmic refinement of the readability conditions used in the Arzhantseva--Ol'shanskii method.  The original complexity estimate in~\cite{KS05} is exponential.  The point below is that, when the rank bound is fixed, a witness graph has uniformly bounded topological complexity, so the labels of its maximal arcs may be chosen among subwords of the input word.  This is analogous to the subword enumeration used for the quartic recognition algorithm in~\cite[Lemma~3.10]{KS09}; we will only need a polynomial bound here.

Here an $A$-labelled graph is an oriented graph whose oriented edges have labels in $A^{\pm1}$, with the reverse orientation carrying the inverse label.  Edge counts below refer to unoriented edges, and the label of a path is obtained by reading its oriented edge labels in order.

\begin{lem}[Polynomial recognition of fixed-rank readability]\label{lem:readability-poly}
Fix $m\ge 2$, an integer $R\ge 1$, and a rational number $0<\mu\le1$.  Let $w$ be a nontrivial freely reduced word over $A^{\pm1}$, of length $\ell$.  Each of the following conditions is decidable in deterministic time polynomial in $\ell$:
\begin{enumerate}[(1)]
\item there exists a finite connected $A$-labelled graph $\Gamma$ with at most $\mu\ell$ unoriented edges and $\rank\pi_1(\Gamma)\le R$ in which $w$ labels a reduced path;
\item there exists such a graph $\Gamma$ which, in addition, has a vertex of degree strictly less than $2m$.
\item For fixed parameters, the $\mu$-readability and $(\mu,L)$-readability conditions of~\cite[Definitions~4.1--4.2]{KS05} are polynomial-time decidable.
\end{enumerate}
\end{lem}

\begin{proof}
Suppose first that a graph $\Gamma$ as in part~\textup{(1)} exists, and let $p$ be a reduced path labelled by $w$.  Replace $\Gamma$ by the connected subgraph $H$ consisting of the edges traversed by $p$.  This does not increase the number of edges or the rank, and every edge of $H$ is traversed by $p$.

Every degree-one vertex of $H$ is an endpoint of $p$: otherwise $p$ would have to enter that vertex and immediately traverse the same edge backwards.  Thus $H$ has at most two degree-one vertices.  If $r=\rank\pi_1(H)\le R$, the identity
\[
 \sum_{x\in V(H)}(\deg(x)-2)=2r-2
\]
shows that the number of vertices of degree at least three is at most $2R$.  Mark the initial and terminal vertices of $p$ and suppress every other degree-two vertex.  The resulting topological graph therefore has a number of vertices and edges bounded by a constant depending only on $R$; for example, it has at most $2R+2$ marked or branching vertices and at most $3R+1$ topological edges.

Choose an orientation on each topological edge.  Once the reduced path $p$ enters the interior of a maximal arc (a path whose internal vertices have degree two), it must continue along that arc until it reaches an endpoint of the arc; the initial and terminal vertices of $p$ have already been marked.  Since every edge of $H$ is traversed by $p$, the label of every oriented topological edge is therefore a nonempty subword of $w$ or $w^{-1}$.  There are $O(\ell^2)$ possible such labels.  Since there are only finitely many topological graph types with marked initial and terminal vertices, and only a bounded number of topological edges depending on $R$ alone, all possible labelled graphs $H$ can therefore be enumerated in polynomial time.

For each candidate, expand the topological edges and check the edge and rank bounds.  A finite-state dynamic program, propagating the possible current vertices while the letters of $w$ are read, decides in polynomial time whether $w$ labels a path in the candidate.  Because $w$ is freely reduced, any path with label $w$ is automatically reduced.  The expanded candidate has $O(\ell)$ edges, so each candidate is checked in polynomial time.  This proves part~\textup{(1)}.

For part~\textup{(2)}, again let $H$ be the union of the edges traversed by $p$.  If $H$ already has a vertex of degree less than $2m$, the preceding enumeration applies unchanged.  Otherwise the required low-degree vertex of $\Gamma$ lies outside $H$, so $H$ is a proper subgraph of the connected graph $\Gamma$.  Hence at least one edge of $\Gamma\setminus H$ is incident to $H$, and therefore
\[
 |E(H)|+1\le |E(\Gamma)|\le\mu\ell.
\]
Attach one pendant edge, with any label, to an arbitrary vertex of $H$.  The resulting connected labelled graph has the same rank as $H$, has at most $\mu\ell$ edges, still contains the path labelled $w$, and has a degree-one vertex.  Thus it suffices, for each candidate $H$ from part~\textup{(1)}, to check either that $H$ already has a vertex of degree less than $2m$ or that the edge bound allows one additional pendant edge.  This is again polynomial.

The definitions of $\mu$-readability and $(\mu,L)$-readability in~\cite{KS05} are precisely the two cases above with fixed rank bounds $m-1$ and $L$, respectively.  In the latter definition the path is not explicitly required to be reduced, but this is automatic when its label is the freely reduced word $w$.
\end{proof}

\begin{thm}[Kapovich--Schupp, with polynomial recognition]\label{thm:KS-input}
Fix $m\ge 2$ and $q\ge 1$.  There are subsets
\[
 \mathcal P_{m,q}(n)\subseteq\CR_n^q
 \qquad(n\ge 1)
\]
and constants $C,c>0$ such that:
\begin{enumerate}[(1)]
\item For every $n\ge 1$,
\[
 \frac{|\mathcal P_{m,q}(n)|}{|\CR_n|^q}
 \ge 1-Ce^{-cn}.
\]
\item For every $n\ge 1$, membership in $\mathcal P_{m,q}(n)$ is decidable in deterministic polynomial time in $n$.
\item For every $n\ge 1$ and every $R=(r_1,\dots,r_q)\in\mathcal P_{m,q}(n)$, put
\[
 G_R=\langle a_1,\dots,a_m\mid r_1,\dots,r_q\rangle.
\]
Then $G_R$ is torsion-free, one-ended, and non-elementary hyperbolic, every subgroup generated by at most $m-1$ elements is free, and every $m$-tuple generating a nonfree subgroup is Nielsen-equivalent to the standard tuple $(\ol a_1,\dots,\ol a_m)$, where $\ol a_i$ denotes the image of $a_i$ in $G_R$.
\item For every $n\ge 1$ and every $R\in\mathcal P_{m,q}(n)$, with $G_R$ as in part~\textup{(3)}, every nonfree subgroup of $G_R$ generated by at most $m$ elements is all of $G_R$, every automorphism of $G_R$ lifts to an automorphism of $F_m$, and $G_R$ is co-Hopfian.
\end{enumerate}
\end{thm}

\begin{proof}
Choose fixed rational parameters $\lambda,\mu$ and an integer $L$ as in the proof of~\cite[Theorem~B]{KS05}, and take the corresponding Kapovich--Schupp class.  The group-theoretic conclusions are those of~\cite[Theorem~B]{KS05}.  Their theorem is stated in the cumulative few-relator model.  If
\[
 B_n^{\mathrm{cr}}=\sum_{1\le\ell\le n}|\CR_\ell|,
\]
then \eqref{eq:number-cyclically-reduced} gives $B_n^{\mathrm{cr}}\asymp|\CR_n|\asymp(2m-1)^n$.  The bad equal-length tuples lie in the bad cumulative set, while the cumulative ambient set has size $(B_n^{\mathrm{cr}})^q\asymp|\CR_n|^q$; hence exponential genericity transfers to the equal-length sphere.

It remains only to improve the recognition bound stated in~\cite{KS05}.  By~\cite[Definition~4.3]{KS05}, membership in the relevant class requires the $C'(\lambda)$ condition, absence of proper-power relators, and the requirement that every subword of a cyclic conjugate of a relator having length at least half that relator length be neither $\mu$-readable nor $(\mu,L)$-readable.  Small cancellation and proper-power detection are polynomial-time decidable.  There are only polynomially many relevant cyclic subwords, and Lemma~\ref{lem:readability-poly} tests each readability condition in polynomial time, with degree depending only on the fixed parameters.  Thus membership in $\mathcal P_{m,q}(n)$ is polynomial-time decidable.  This is the same finite-subword principle that yields the explicit quartic recognition bound in the modular-group setting of~\cite[Theorem~A and Lemma~3.10]{KS09}; no particular polynomial degree is needed here.

Part~\textup{(4)} follows as in~\cite{KS05}: Nielsen moves preserve generated subgroups, automorphisms send the standard tuple to generating tuples, and an injective endomorphism has nonfree image generated by at most $m$ elements.
\end{proof}

\subsection{Generic Whitehead input}

A cyclically reduced word $w$ is \emph{strictly minimal} if every non-inner Whitehead automorphism of the second kind strictly increases its cyclic length.  Let $\mathcal Z_m$ be the class denoted $Z$ in~\cite[Definition~8.6]{KSS06}.  Thus its elements are strictly minimal and root-free, no nontrivial relabeling sends one to a conjugate of itself, and no relabeling sends one to a conjugate of its inverse. The following result is due to Kapovich--Schupp--Shpilrain~\cite{KSS06}:

\begin{thm}[Kapovich--Schupp--Shpilrain]\label{thm:KSS-input}
Fix $m\ge 2$.  Then the following hold.
\begin{enumerate}[(1)]
\item The subsets $\mathcal Z_m\cap\CR_n$ are exponentially generic in $\CR_n$ as $n\to\infty$, and membership in $\mathcal Z_m$ is decidable in linear time.
\item If $w\in\mathcal Z_m$, then
\[
 \Stab_{\Out(F_m)}([w])=\{1\}.
\]
\item If $u$ is strictly minimal and $v$ is cyclically reduced with $|u|=|v|$, then $[u]$ and $[v]$ lie in the same $\Out(F_m)$-orbit if and only if a relabeling automorphism sends $u$ to a cyclic shift of $v$.
\item On any pair of cyclically reduced input words for which at least one input lies in the automorphic orbit of a strictly minimal word, Whitehead's algorithm, in search form, runs in at most quadratic time in the total input length.  If both inputs are strictly minimal, orbit testing is linear.
\end{enumerate}
\end{thm}

\begin{proof}
Theorem~B, Theorem~8.5, and Proposition~8.7 of~\cite{KSS06} give genericity and the stabilizer statement.  Sphere and cumulative exponential genericity are equivalent because $|\CR_n|\asymp(2m-1)^n$ and the cumulative ball has comparable size.  Linear membership tests strict minimality and finitely many relabeling and inverse-conjugacy conditions.  Theorem~A(3)--(5) gives the orbit and complexity claims; recording the Whitehead moves yields the search form.
\end{proof}

\section{Greendlinger normal-closure rigidity}\label{sec:normal-closure}

For a tuple $R=(r_1,\dots,r_q)$ of elements of a free group $F=F(X)$, denote
\[
 N_R=\Normal{r_1,\dots,r_q}\unlhd F,
\]
the normal closure of $R$ in $F$.
Similarly, for a set $\mathcal R\subseteq F$, write $N_{\mathcal R}=\Normal{\mathcal R}\unlhd F$.

We use a modern formulation, Theorem~\ref{thm:Greendlinger-normal} of Greendlinger's classic analogue~\cite{Gre61} of Magnus's theorem.  The original statement uses terminology predating standard small-cancellation language.  The formulation below follows from Greendlinger's argument; see also Proposition~6.5 of~\cite{KS09} for a similar statement for small-cancellation quotients of the modular group.

For completeness, we give a self-contained proof of Theorem~\ref{thm:Greendlinger-normal} using modern small cancellation theory tools.  We use the notation and terminology of Chapter~V of~\cite{LS77}.  There a planar map $M$ is called a \emph{$(3,6)$-map} if every vertex $v$ has degree $d(v)\ge3$ and every interior region $D$ has degree $d(D)\ge6$.  For a boundary region $D$, the quantity $i(D)$ denotes its interior degree, that is, the number of interior edges in the boundary cycle $\partial D$.  We use the standard combinatorial-disc convention that each region is a closed $2$-cell equipped with its boundary cycle; intersections and consecutive arcs below are understood on these boundary cycles.  To a van Kampen diagram $\Delta$ over a symmetrized presentation $\langle X\mid\mathcal R\rangle$, one associates a connected and simply connected planar map $M$ by suppressing degree-two vertices.  Whereas the oriented edges of $\Delta$ are labelled by elements of $X^{\pm1}$, the edges of $M$ are labelled by nontrivial freely reduced words in $F(X)$.

Lemma~2.2 of Chapter~V of~\cite{LS77} implies that the planar map associated to a reduced van Kampen diagram over a $C'(1/6)$ presentation is a $(3,6)$-map.

The following consequence is likely standard, but we include a proof because we have not found this precise formulation in the literature.
\begin{lem}[Region length versus boundary length]\label{lem:diagram}
Let $\Delta$ be a reduced van Kampen diagram over a $C'(1/6)$ presentation $\langle X\mid\mathcal R\rangle$.  Suppose that $\Delta$ contains at least two regions and has no vertices of degree one.  Let
\[
 w=w(\Delta)\in(X\cup X^{-1})^\ast
\]
be the word, not necessarily freely reduced, labelling the boundary cycle $\partial\Delta$.

Then, for every region $D$ of $\Delta$, if $\mu(D)$ is the label of $\partial D$, then
\[
 |\mu(D)|<|w(\Delta)|.
\]
\end{lem}

\begin{proof}
We induct on the number $t$ of regions of $\Delta$.

Suppose first that $t=2$, and denote the two regions by $D$ and $D'$.  If $D$ and $D'$ have no common boundary edge, then the boundary walk of $\Delta$ traverses the entire boundary cycles of both regions and may also traverse one-dimensional connecting portions.  Hence
\[
 |w(\Delta)|\ge |\mu(D)|+|\mu(D')|.
\]
Thus the required strict inequality holds for both regions.

Suppose that $D$ and $D'$ have a common boundary edge.  Since $\Delta$ is simply connected, their common boundary edges form a single arc: two disjoint common arcs would create a nontrivial loop in the union of the two closed regions.  Let $v$ be its label.  Since $\Delta$ is reduced, $v$ is a piece.  Therefore
\[
 |v|<\frac{|\mu(D)|}{6}
 \qquad\text{and}\qquad
 |v|<\frac{|\mu(D')|}{6}.
\]
Moreover,
\[
 |w(\Delta)|=|\mu(D)|+|\mu(D')|-2|v|.
\]
It follows that
\[
 |w(\Delta)|-|\mu(D)|
   =|\mu(D')|-2|v|>0,
\]
and, symmetrically,
\[
 |w(\Delta)|-|\mu(D')|
   =|\mu(D)|-2|v|>0.
\]
This proves the result when $t=2$.

Now let $t>2$ and assume that the conclusion holds for every reduced van Kampen diagram satisfying the hypotheses of the lemma and having $t'$ regions, where $2\le t'<t$.

If $\Delta$ is not a disc diagram, cut the planar complex at its cut vertices and consider the maximal disc subdiagrams containing regions.  Distinct such disc components meet, if at all, only at cut vertices, and the boundary walk of $\Delta$ is obtained by concatenating their boundary walks together with possible one-dimensional connecting portions, which are traversed twice.  Since there is more than one disc component, every component containing a region has fewer than $t$ regions, and
\[
 |w(\Delta)|\ge \sum_{\Delta'} |w(\Delta')|.
\]
Let $D$ be a region contained in a disc component $\Delta'$.  If $\Delta'$ has one region, then
\[
 |\mu(D)|\le |w(\Delta')|<|w(\Delta)|.
\]
If $\Delta'$ has at least two regions, the inductive hypothesis gives
\[
 |\mu(D)|<|w(\Delta')|\le |w(\Delta)|.
\]
Thus the result follows, and we may assume that $\Delta$ is a disc diagram.

Let $M$ be the planar map associated to $\Delta$.  Since $M$ is a $(3,6)$-map, Theorem~4.3 of Chapter~V of~\cite{LS77} gives
\[
 \sum_M^\ast [4-i(D)]\ge6,
\]
where the sum is taken over the boundary regions $D$ for which $\partial D\cap\partial M$ is a consecutive part of $\partial M$.  Since $t>1$, every such region has $i(D)\ge 1$, and hence every positive summand is at most $3$.  Hence at least two distinct boundary regions $D_1,D_2$ occur in this sum with
\[
 1\le i(D_1)\le3,
 \qquad
 1\le i(D_2)\le3.
\]

We first remove only the region $D_1$ from $\Delta$.  Let $p_1$ be the total length, in the original $X^{\pm1}$-labelling, of the interior part of $\partial D_1$, and let $e_1$ be the length of the exterior part $\partial D_1\cap\partial\Delta$.  Every edge of the associated map contained in the interior part of $\partial D_1$ is labelled by a piece.  Hence
\[
 p_1<i(D_1)\frac{|\mu(D_1)|}{6}
     \le\frac{|\mu(D_1)|}{2}.
\]
Since $e_1=|\mu(D_1)|-p_1$, we have $e_1>p_1$.

Because $\partial D_1\cap\partial M$ is a consecutive part of the boundary cycle of $M$, its complementary interior part on the boundary cycle of $D_1$ is also a single consecutive arc.  Thus $D_1$ is a boundary ear: it is attached to the closure of the remaining disc diagram precisely along this interior arc.  Deleting the open $2$-cell $D_1$ together with the open exterior boundary arc therefore leaves a connected and simply connected subdiagram with $t-1$ regions. We denote this subdiagram by $\Delta_1$. On the boundary walk, this operation replaces the exterior arc of length $e_1$ by the interior attaching arc of length $p_1$.  Thus its boundary length is
\[
 |w(\Delta)|-e_1+p_1<|w(\Delta)|.
\]
By construction, $\Delta_1$ is a disc van Kampen diagram with no degree-one vertices (although the label of the boundary cycle of $\Delta_1$ is not necessarily cyclically reduced even if the label of the boundary cycle of $\Delta$ is cyclically reduced.).  Since $\Delta$ was a reduced diagram, $\Delta_1$ is also reduced. 
The digram $\Delta_1$ has exactly $t-1\ge 2$ regions, contains every region of $\Delta$ other than $D_1$, and satisfies
\[
 |w(\Delta_1)|<|w(\Delta)|.
\]
Thus the inductive hypothesis applied to $\Delta_1$ and gives, for every region $D\ne D_1$ of $\Delta$,
\[
 |\mu(D)|<|w(\Delta_1)|<|w(\Delta)|.
\]
In particular, this estimate holds for $D_2$.

A symmetric argument with removing $D_2$ from $\Delta$ yields $|\mu(D_1)|<|w(\Delta)|$, completing the inductive step.
\end{proof}

The argument permits intermediate boundary labels that are not freely reduced, even when the original boundary label is freely and cyclically reduced.

\begin{thm}[Greendlinger-type normal-closure rigidity]\label{thm:Greendlinger-normal}
Let $\mathcal R$ and $\mathcal S$ be symmetrized sets of nontrivial cyclically reduced words in a free group $F=F(X)$, and suppose that both $\mathcal R$ and $\mathcal S$ satisfy $C'(1/6)$.  If
\[
 N_{\mathcal R}=N_{\mathcal S},
\]
then
\[
 \mathcal R=\mathcal S.
\]
\end{thm}

\begin{proof}
Suppose $\mathcal R\ne\mathcal S$, and choose a shortest element $z$ of the symmetric difference $\mathcal R\mathbin\triangle\mathcal S$.  After interchanging the two sets, assume that $z\in\mathcal R\setminus\mathcal S$.  Since $z\in N_{\mathcal S}$, there is a reduced van Kampen disc diagram $\Delta$ over $\langle X\mid\mathcal S\rangle$ whose boundary label is $z$.  The word $z$ is cyclically reduced, so $\Delta$ has no boundary spurs and hence no vertices of degree one.  Moreover, $\Delta$ has at least two regions: a one-region diagram would imply that $z\in\mathcal S$, since $\mathcal S$ is symmetrized.

By Lemma~\ref{lem:diagram}, every region label $s\in\mathcal S$ occurring in $\Delta$ satisfies $|s|<|z|$.  The minimal choice of $z$ therefore implies that every such $s$ belongs to $\mathcal R\cap\mathcal S$.

Greendlinger's lemma~\cite[Chapter~V, Theorem~4.4]{LS77} gives a boundary region of $\Delta$ whose boundary shares with $\partial\Delta$ a segment labelled by a word $v$ satisfying
\[
 |v|>\frac{|s|}{2},
\]
where $s\in\mathcal R\cap\mathcal S$ is the label of that region.  After taking suitable cyclic conjugates, and inverses if necessary, two elements of $\mathcal R$, one of length $|s|$ and the other of length $|z|$, have $v$ as a common initial segment.  These elements are distinct because $|s|<|z|$.  Thus $v$ is a piece for $\mathcal R$, contradicting
\[
 |v|>\frac{|s|}{2}>\frac{|s|}{6}
\]
and the $C'(1/6)$ condition for $\mathcal R$.
\end{proof}

The sets $X$, $\mathcal R$, and $\mathcal S$ in Theorem~\ref{thm:Greendlinger-normal} need not be finite.

\begin{cor}[Irredundant tuple form]\label{cor:Greendlinger-tuples}
Let $q,t\ge 1$ and $0<\lambda\le1/6$.  Let
\[
 R=(r_1,\dots,r_q),
 \qquad
 S=(s_1,\dots,s_t)
\]
be irredundant tuples of nontrivial elements of $F_m$.  Suppose that the symmetrizations of cyclically reduced forms of both tuples satisfy $C'(\lambda)$.  If
\[
 N_R=N_S,
\]
then $q=t$, and after a permutation of the indices every $r_i$ is conjugate in $F_m$ to some $s_j^{\pm1}$.
\end{cor}

\begin{proof}
Choose cyclically reduced conjugates $u_i$ of the $r_i$ and $v_j$ of the $s_j$, and let $\mathcal R$ and $\mathcal S$ be their symmetrizations.  Since $\lambda\le1/6$, both sets satisfy $C'(1/6)$, and
\[
 N_{\mathcal R}=N_R=N_S=N_{\mathcal S}.
\]
Theorem~\ref{thm:Greendlinger-normal} gives $\mathcal R=\mathcal S$.  The cyclic-permutation-and-inversion orbits in a symmetrized set are its symmetrized relator classes.  Irredundancy says that the entries of $R$ and $S$ represent, respectively, $q$ and $t$ distinct such classes.  Hence $q=t$, and the asserted correspondence follows.
\end{proof}

\begin{rem}
Root-freeness is not needed in Theorem~\ref{thm:Greendlinger-normal} or Corollary~\ref{cor:Greendlinger-tuples}.  Irredundancy is needed only to recover the number of entries in a tuple: repetitions or duplicate conjugacy classes do not change a normal closure.  Greendlinger's shortest-unmatched-relator argument removes any common-length requirement.
\end{rem}

\section{Isomorphism rigidity}\label{sec:isomorphism}

Let $F_m=F(A)$ be as in Convention~\ref{conv:standing}.  Throughout this section, $R,S$ are finite nontrivial tuples with normal closures $N_R,N_S$.  We first isolate the lifting consequence of Nielsen uniqueness.

\begin{lem}[Lifting an isomorphism to the free group]\label{lem:isomorphism-lift}
Let $R$ and $S$ be finite tuples of nontrivial elements of $F_m$, and let $N_R,N_S\unlhd F_m$ be their normal closures.  Put
\[
 G_R=F_m/N_R,
 \qquad
 G_S=F_m/N_S.
\]
Suppose that every generating $m$-tuple of $G_R$ is Nielsen-equivalent to the standard tuple $(\ol a_1,\dots,\ol a_m)$, where $\ol a_i$ denotes the image of $a_i$ in $G_R$.  If $G_R\cong G_S$, then there exists $\Phi\in\Aut(F_m)$ such that
\[
 \Phi(N_R)=N_S.
\]
\end{lem}

\begin{proof}
Let $\pi_R:F_m\to G_R$ and $\pi_S:F_m\to G_S$ be the quotient maps and $\Psi:G_R\to G_S$ an isomorphism.  The tuple
\[
 \bigl(\Psi^{-1}(\pi_S(a_1)),\dots,
       \Psi^{-1}(\pi_S(a_m))\bigr)
\]
generates $G_R$.  By Nielsen uniqueness, there is $\Theta\in\Aut(F_m)$ such that
\[
 \Psi^{-1}\circ\pi_S=\pi_R\circ\Theta.
\]
Taking kernels gives
\[
 N_S=\Theta^{-1}(N_R).
\]
Thus $\Phi=\Theta^{-1}$ satisfies $\Phi(N_R)=N_S$.
\end{proof}

\begin{prop}[Nielsen--Greendlinger isomorphism rigidity]\label{prop:Nielsen-Greendlinger}
Let $0<\lambda\le1/6$, and let
\[
 R=(r_1,\dots,r_q)
\]
be an irredundant $\lambda$-stable tuple of nontrivial elements of $F_m$.  Let $N_R\unlhd F_m$ be its normal closure and put
\[
 G_R=F_m/N_R.
\]
Suppose that every generating $m$-tuple of $G_R$ is Nielsen-equivalent to the standard tuple $(\ol a_1,\dots,\ol a_m)$, where $\ol a_i$ denotes the image of $a_i$ in $G_R$.  Let
\[
 S=(s_1,\dots,s_t)
\]
be an irredundant tuple of nontrivial elements of $F_m$ such that the symmetrization of cyclically reduced representatives of its entries satisfies $C'(\lambda)$.  Let $N_S\unlhd F_m$ be its normal closure and put
\[
 G_S=F_m/N_S.
\]
Then
\[
 G_R\cong G_S
\]
if and only if $t=q$ and there exist $\Phi\in\Aut(F_m)$, a permutation $\sigma\in\Sym(q)$, signs $\epsilon_i\in\{\pm1\}$, and elements $g_i\in F_m$ such that
\[
 \Phi(r_i)=g_i s_{\sigma(i)}^{\epsilon_i}g_i^{-1}
 \qquad(1\le i\le q).
\]
\end{prop}

\begin{proof}
If the displayed relations hold, then $\Phi(N_R)=N_S$, so $\Phi$ induces an isomorphism $G_R\cong G_S$.

Conversely, suppose $G_R\cong G_S$.  Lemma~\ref{lem:isomorphism-lift} gives $\Phi\in\Aut(F_m)$ such that
\[
 \Phi(N_R)=N_S.
\]
Choose cyclically reduced conjugates $u_i$ of $\Phi(r_i)$ and $v_j$ of $s_j$, and put
\[
 U=(u_1,\dots,u_q),
 \qquad
 V=(v_1,\dots,v_t).
\]
Then $N_U=N_V$.  The tuple $U$ is irredundant, and $\lambda$-stability of $R$ gives $C'(\lambda)$ for its symmetrization; the hypothesis on $S$ gives the same for $V$.  Corollary~\ref{cor:Greendlinger-tuples} therefore gives $t=q$ and, after a permutation, makes every $u_i$ a cyclic conjugate of $v_{\sigma(i)}^{\epsilon_i}$.  Undoing the conjugations defining $u_i$ and $v_j$ gives the required elements $g_i$.
\end{proof}

\section{The generic rigidity class}\label{sec:generic-class}

For a cyclic word $w$, write $\sh_t(w)$ for cyclic shift by $t$ positions.

Fix $m\ge 2$ and $q\ge 1$.  For $n\ge3$, let $\Gamma_{m,q,n}$ be the group of transformations of $\CR_n^q$ generated by:
\begin{enumerate}[(i)]
\item one global relabeling $\Theta\in\Rel(A)$;
\item a permutation of the $q$ components;
\item independent inversion of the components;
\item independent cyclic shifts of the components.
\end{enumerate}
Thus every element has the form
\begin{equation}\label{eq:Gamma-action}
 (r_1,\dots,r_q)
 \longmapsto
 \left(
 \sh_{t_i}\bigl(\Theta(r_{\sigma(i)})^{\epsilon_i}\bigr)
 \right)_{i=1}^q,
\end{equation}
where $\sigma\in\Sym(q)$, $\epsilon_i\in\{\pm1\}$, and $t_i\in\Z/n\Z$.

\begin{lem}[The obvious symmetry group]\label{lem:Gamma-order}
Fix $m\ge 2$, $q\ge 1$, and an integer $n\ge3$.  Then the action of $\Gamma_{m,q,n}$ on $\CR_n^q$ defined in \eqref{eq:Gamma-action} is faithful and
\begin{equation}\label{eq:Gamma-order}
 |\Gamma_{m,q,n}|
 =2^{m+q}m!\,q!\,n^q.
\end{equation}
\end{lem}

\begin{proof}
The rotations and inversions of the individual cyclic words form a product of $q$ dihedral groups of order $2n$; component permutations permute these factors, while the global relabeling group has order $2^m m!$ and commutes with rotations and component permutations.

To see faithfulness, suppose that a transformation acts identically on all tuples.  Varying the components independently first forces the component permutation to be trivial.  On each component the resulting map has the form
\[
 w\longmapsto\sh_t\bigl(\Theta(w)^\epsilon\bigr).
\]
Applying it to the constant cyclic words $a^n$, for all $a\in A^{\pm1}$, shows, on each component, that either $\epsilon=1$ and $\Theta=1$, or $\epsilon=-1$ and $\Theta$ is the global inversion relabeling.  Since $\Theta$ is global, the same alternative holds on every component.  In the first case a positive cyclic word with no nontrivial rotational period forces every $t_i=0$.  In the second case the map on the $i$-th component is a prescribed reflection of the $n$ cyclic positions.  Since $n\ge3$, choose two positions interchanged by that reflection, put different positive basis letters at those positions, and fill all remaining positions with one of them.  The resulting word is aperiodic and is not fixed by the prescribed reflection.  Hence the second case is impossible.  Thus the action is faithful, and the semidirect-product description gives \eqref{eq:Gamma-order}.
\end{proof}

Let $\mathcal P_{m,q}(n)$ be the Kapovich--Schupp class from Theorem~\ref{thm:KS-input}.  Replace it by its $\Gamma_{m,q,n}$-invariant core
\[
 \mathcal P_{m,q}^{\mathrm{inv}}(n)
 =\bigcap_{\gamma\in\Gamma_{m,q,n}}
   \gamma\mathcal P_{m,q}(n).
\]
Since $|\Gamma_{m,q,n}|$ is polynomial in $n$, this invariant core remains exponentially generic and polynomial-time recognizable.

Fix
\[
 \lambda_0=\frac{1}{12}.
\]
The uniform current $\nu_A$ is invariant under $\Rel(A)$.  By part~\textup{(1)} of Lemma~\ref{lem:uniform-filling-neighborhood}, choose a $\Rel(A)$-invariant neighborhood
\[
 \mathcal U\subseteq\mathcal S_A
\]
of $\nu_A$ and a rational number $0<c\le1$ such that
\begin{equation}\label{eq:generic-current-lower}
 \langle T,\eta\rangle\ge c\Lambda_A(T)
 \qquad(\eta\in\mathcal U,\,T\in\cvbar_m\setminus\{0\}).
\end{equation}
We may and do choose $\mathcal U$ to be defined by finitely many strict rational cylinder-coordinate inequalities.  Indeed, first choose an invariant open filling neighborhood, then choose a neighborhood specified by finitely many rational cylinder-coordinate inequalities whose closure lies inside it, and intersect its finitely many relabeling translates.  These rational inequalities and the constant $c$ are fixed auxiliary data for the classes below; the membership statement does not require a uniform procedure that constructs this data from $m$ and $q$.  Accordingly, when we call the classes $\calQ_{m,q}(n)$ \emph{finitely specified}, we mean that for fixed $m,q$ this finite rational auxiliary data has been chosen once and for all.

For $r\in\CR_n$, put
\[
 \widehat\eta_r=\frac{1}{n}\eta_r\in\mathcal S_A.
\]

\begin{defn}[The class $\calQ_{m,q}(n)$]\label{def:Qcr}
For each integer $n\ge 1$, define $\calQ_{m,q}(n)$ as follows.  If $1\le n<3$, put $\calQ_{m,q}(n)=\varnothing$.  If $n\ge3$, a tuple
\[
 R=(r_1,\dots,r_q)\in\CR_n^q
\]
belongs to $\calQ_{m,q}(n)$ if:
\begin{enumerate}[(Q1)]
\item $R\in\mathcal P_{m,q}^{\mathrm{inv}}(n)$;
\item $\widehat\eta_{r_i}\in\mathcal U$ for every $i$;
\item
\begin{equation}\label{eq:Qcr-overlap}
 2\bigl(\rhoSigned(R)+2m\bigr)<c\lambda_0 n;
\end{equation}
\item every $r_i$ belongs to $\mathcal Z_m$, and
\[
 \Stab_{\Gamma_{m,q,n}}(R)=\{1\}.
\]
\end{enumerate}
\end{defn}

\begin{lem}[Deterministic consequences]\label{lem:Qcr-consequences}
Fix $n\ge 1$ and let $R\in\calQ_{m,q}(n)$.  Let $N_R\unlhd F_m$ be the normal closure of $R$ and put $G_R=F_m/N_R$.  Then:
\begin{enumerate}[(1)]
\item $R$ is admissible and $\lambda_0$-stable;
\item every component $r_i$ is strictly minimal and satisfies $\Stab_{\Out(F_m)}([r_i])=\{1\}$;
\item the conclusions of parts~\textup{(3)}--\textup{(4)} of Theorem~\ref{thm:KS-input} hold for $G_R$.
\end{enumerate}
\end{lem}

\begin{proof}
Condition (Q3) makes $\rhoSigned(R)$ finite, so Lemma~\ref{lem:rho-signed-admissible} gives admissibility.  For $\phi\in\Out(F_m)$, condition (Q2) and \eqref{eq:generic-current-lower} give
\[
 \|\phi(r_i)\|_A
 =n\langle T_A\phi,\widehat\eta_{r_i}\rangle
 \ge cn\Lambda_A(\phi).
\]
If $v$ is a piece beginning a relator arising from component $i$, Proposition~\ref{prop:deterministic-piece} and (Q3) give
\[
 |v|
 \le2\Lambda_A(\phi)(\rhoSigned(R)+2m)
 <\lambda_0cn\Lambda_A(\phi)
 \le\lambda_0\|\phi(r_i)\|_A.
\]
This proves (1).  Part (2) follows from (Q4) and Theorem~\ref{thm:KSS-input}, while (3) follows from (Q1) and Theorem~\ref{thm:KS-input}.
\end{proof}

\begin{prop}[Genericity and effectiveness]\label{prop:Qcr-generic}
Fix $m\ge 2$ and $q\ge 1$, and let $\calQ_{m,q}(n)$ be the classes from Definition~\ref{def:Qcr}.  Then the following hold.
\begin{enumerate}[(1)]
\item There are constants $C,c_1>0$ such that, for every $n\ge 1$,
\[
 \frac{|\calQ_{m,q}(n)|}{|\CR_n|^q}
 \ge 1-Ce^{-c_1n}.
\]
\item For every $n\ge 1$, membership in $\calQ_{m,q}(n)$ is decidable in deterministic polynomial time in $n$.
\end{enumerate}
\end{prop}

\begin{proof}
Condition (Q1) is exponentially generic by Theorem~\ref{thm:KS-input} and the polynomial-size invariant-core construction.  Lemma~\ref{lem:cr-current-concentration} and a finite union bound treat (Q2).  Failure of (Q3) implies
\[
 \rhoSigned(R)\ge\frac{c\lambda_0}{2}n-2m,
\]
so Proposition~\ref{prop:signed-overlap-tail} treats (Q3).  The class $\mathcal Z_m$ is exponentially generic by Theorem~\ref{thm:KSS-input}.

It remains to treat tuple symmetries.  If every component lies in $\mathcal Z_m$, no nonidentity element of $\Gamma_{m,q,n}$ with trivial component permutation can fix the tuple: the defining properties of $\mathcal Z_m$ exclude nontrivial relabeling and inverse-conjugacy, while root-freeness excludes nonzero cyclic rotations.  If the component permutation is nontrivial, a fixed tuple satisfies, for some $i\ne j$, an equality between $r_i$ and a relabeling, inversion, and cyclic shift of $r_j$.  Since these operations preserve the uniform measure on $\CR_n$, the probability of any prescribed such equality is $|\CR_n|^{-1}$.  A union bound over the polynomially many elements of $\Gamma_{m,q,n}$ gives an exponentially small probability.

For effectiveness, Theorem~\ref{thm:KS-input}(2) and the polynomial-size invariant-core construction show that (Q1) is polynomial-time decidable for fixed $m,q$.  The finite cylinder inequalities in (Q2), signed periodic overlaps in (Q3), membership in $\mathcal Z_m$, and all finite symmetry checks are polynomial-time decidable as well.  Thus the total membership test is polynomial-time.
\end{proof}

\begin{lem}[Essential cyclic splittings]\label{lem:essential-cyclic-splitting}
Let $G$ be a torsion-free hyperbolic group.  Suppose that $G$ admits a nontrivial graph-of-groups splitting over infinite cyclic edge groups such that every edge group has infinite index in each incident vertex group.  Then $\Out(G)$ contains an element of infinite order.
\end{lem}

\begin{proof}
Collapse all but one edge orbit in the Bass--Serre tree.  The resulting one-edge splitting retains the infinite-index condition and is either an amalgam
\[
 G=G_1\mathbin{\ast_C}G_2
\]
or an HNN extension
\[
 G=\langle H,t\mid t^{-1}C_-t=C_+\rangle,
\]
where the displayed edge groups are infinite cyclic.  In particular, $G_1,G_2$, and $H$ are noncyclic.

In the amalgam case, choose $1\ne c\in C$ and let $\delta$ fix $G_1$ pointwise and act on $G_2$ by conjugation by $c$.  This is the standard Dehn twist.  If $\delta^k$ were inner for some $k\ne0$, say $\delta^k=\operatorname{ad}(h)$, then $h$ would centralize $G_1$.  Hence $h=1$: if $h\ne1$, then the cyclic centralizer $C_G(h)$ would contain the noncyclic group $G_1$.  Thus $\delta^k$ would be the identity.  But $G_2$ is noncyclic, whereas $C_G(c^k)$ is cyclic, so some $g\in G_2$ does not commute with $c^k$ and
\[
 \delta^k(g)=c^kgc^{-k}\ne g,
\]
a contradiction.

In the HNN case, choose $1\ne c\in C_-$ and define $\delta$ by $\delta|_H=\operatorname{id}$ and $\delta(t)=ct$.  Since $C_-$ is cyclic, this preserves the defining relation.  If a nonzero power $\delta^k$ were inner, its conjugating element would centralize the noncyclic group $H$ and hence would be trivial.  This would make $\delta^k$ the identity, contrary to
\[
 \delta^k(t)=c^kt\ne t.
\]
Thus the outer class of $\delta$ has infinite order.  This is the standard Dehn-twist argument; compare~\cite[Theorem~1.2]{GL15}.
\end{proof}

\subsection{Proof of generic few-relator rigidity}

\phantomsection\label{proof:mainthm-rigidity}
\begin{mainthmrestated}{mainthm:rigidity}{Generic few-relator rigidity}
Fix $m\ge 2$ and $q\ge 1$.  For a finite tuple $R=(r_1,\dots,r_s)$ of elements of $F_m$, write
\[
 G_R=\langle a_1,\dots,a_m\mid r_1,\dots,r_s\rangle,
\]
and denote by $\ol a_i$ the image of $a_i$ in $G_R$.
There are finitely specified subsets
\[
 \calQ_{m,q}(n)\subseteq\CR_n^q
 \qquad(n\ge 1)
\]
and constants $C,c>0$ such that:
\begin{enumerate}[(1)]
\item For every $n\ge 1$, a uniform random tuple in $\CR_n^q$ belongs to $\calQ_{m,q}(n)$ with probability at least $1-Ce^{-cn}$.  Membership in $\calQ_{m,q}(n)$ is decidable in deterministic polynomial time in $n$.
\item Every $q$-tuple $R\in\calQ_{m,q}(n)$ is irredundant, root-free, and $1/6$-stable.
\item If $R=(r_1,\dots,r_q)\in\calQ_{m,q}(n)$, then $G_R$ is torsion-free, one-ended, and non-elementary hyperbolic; every subgroup generated by at most $m-1$ elements is free; every $m$-tuple generating a nonfree subgroup is Nielsen-equivalent to the standard tuple $(\ol a_1,\dots,\ol a_m)$; and $G_R$ is co-Hopfian, meaning that every injective endomorphism of $G_R$ is surjective.
\item If $R=(r_1,\dots,r_q)\in\calQ_{m,q}(n)$, then $G_R$ is complete, meaning that its center is trivial and every automorphism of $G_R$ is inner.  Its boundary is one-dimensional, connected, and has no local cut points, and hence is homeomorphic to the Menger curve or the Sierpi\'nski carpet.  If $q\ge m-1$, then $\partial G_R$ is the Menger curve.
\item Let $R=(r_1,\dots,r_q)\in\calQ_{m,q}(n)$, and let $S=(s_1,\dots,s_t)$ be an irredundant tuple of nontrivial elements of $F_m$ whose symmetrization of cyclically reduced representatives satisfies $C'(1/6)$.  Then $G_R\cong G_S$ if and only if $t=q$ and there are $\Phi\in\Aut(F_m)$, $\sigma\in\Sym(q)$, signs $\epsilon_i\in\{\pm1\}$, and $g_i\in F_m$ such that
\[
 \Phi(r_i)=g_i s_{\sigma(i)}^{\epsilon_i}g_i^{-1}
 \qquad(1\le i\le q).
\]

\item If $R=(r_1,\dots,r_q)\in\calQ_{m,q}(n)$ and $S=(s_1,\dots,s_q)\in\calQ_{m,q}(n')$, then $G_R\cong G_S$ if and only if $n=n'$ and there are $\Theta\in\Rel(A)$, $\sigma\in\Sym(q)$, and signs $\epsilon_i\in\{\pm1\}$ such that, for each $i=1,\dots,q$, the word $s_i$ is a cyclic permutation of $\Theta(r_{\sigma(i)})^{\epsilon_i}$. 
\end{enumerate}
\end{mainthmrestated}

\begin{proof}
Proposition~\ref{prop:Qcr-generic} gives part~(1).  Lemma~\ref{lem:Qcr-consequences}(1) gives part~(2): admissibility means root-freeness and irredundancy, and $\lambda_0$-stability with $\lambda_0=1/12$ implies $1/6$-stability.  Lemma~\ref{lem:Qcr-consequences} together with Theorem~\ref{thm:KS-input} gives part~(3).

We prove completeness.  Let $R=(r_1,\dots,r_q)\in\calQ_{m,q}(n)$ and let $\Psi\in\Aut(G_R)$.  By Nielsen uniqueness, $\Psi$ lifts to $\Phi\in\Aut(F_m)$ with $\Phi(N_R)=N_R$.  Stability and Corollary~\ref{cor:Greendlinger-tuples} give a permutation $\sigma$, signs $\epsilon_i$, and conjugators $g_i$ such that
\begin{equation}\label{eq:auto-relator-permutation}
 \Phi(r_i)=g_i r_{\sigma(i)}^{\epsilon_i}g_i^{-1}.
\end{equation}
The words $r_1$ and $r_{\sigma(1)}^{\epsilon_1}$ are strictly minimal and have the same length.  Theorem~\ref{thm:KSS-input}(3) gives a relabeling $\Theta\in\Rel(A)$ carrying $r_1$ to a cyclic shift of $r_{\sigma(1)}^{\epsilon_1}$.  Hence $\Theta^{-1}\Phi$ fixes $[r_1]$, and the trivial outer stabilizer of $r_1$ gives
\[
 [\Phi]=[\Theta]\in\Out(F_m).
\]
Using this in \eqref{eq:auto-relator-permutation} shows that the global relabeling $\Theta$, the component permutation $\sigma$, the signs $\epsilon_i$, and suitable cyclic shifts define an element of $\Gamma_{m,q,n}$ fixing $R$.  Condition (Q4) makes that element trivial.  Therefore $[\Phi]=1$, so every automorphism of $G_R$ is inner.  The center of a non-elementary torsion-free hyperbolic group is trivial, and hence $G_R$ is complete.

Because $R$ is root-free and satisfies $C'(\lambda_0)$ with $\lambda_0<1/6$, the classical asphericity theorem for $C(6)$ presentations without proper-power relators~\cite[Chapter~V, Theorem~13.3]{LS77} makes the presentation complex a finite two-dimensional $K(G_R,1)$, with
\begin{equation}\label{eq:Euler-cr-presentation}
 \chi(G_R)=1-m+q.
\end{equation}
Since $G_R$ is one-ended and nonfree, Stallings--Swan gives $\cdim(G_R)=2$, and Bestvina--Mess~\cite{BM91} gives $\dim\partial G_R=1$.

By Lemma~\ref{lem:essential-cyclic-splitting} and completeness, $G_R$ admits no essential splitting over an infinite cyclic subgroup, where ``essential'' means that every edge group has infinite index in each incident vertex group.

Since $G_R$ is one-ended, it has no splitting over a finite subgroup.  As it is torsion-free, its two-ended subgroups are infinite cyclic, so it admits no essential splitting over a two-ended subgroup.  Moreover, $G_R$ is not virtually cocompact Fuchsian, since in the torsion-free case this would make it a closed surface group, whose outer automorphism group is infinite.  The boundary of a one-ended hyperbolic group is connected and locally connected.  Bowditch's construction underlying~\cite[Theorem~6.2]{Bow98} shows that a local cut point would yield a nontrivial splitting over a two-ended subgroup.  In the standard reduction of the canonical graph of groups, each cyclic vertex of this type has valence at least two; collapsing one incident edge absorbs that vertex into a non-two-ended vertex group while leaving another edge, and the resulting splitting is nontrivial and essential in the above sense.  Hence $\partial G_R$ has no local cut points.  Since $\dim\partial G_R=1$, Kapovich--Kleiner~\cite[Theorem~1]{KK00} gives the circle/Menger-curve/Sierpi\'nski-carpet alternatives.  The circle case is excluded by the preceding non-Fuchsian observation, leaving the Menger curve or the Sierpi\'nski carpet.  If $q\ge m-1$, then \eqref{eq:Euler-cr-presentation} gives $\chi(G_R)\ge0$, whereas a torsion-free hyperbolic group with Sierpi\'nski-carpet boundary has negative Euler characteristic by~\cite[Corollary~8]{KK00}.  This proves part~(4).

For part~(5), Lemma~\ref{lem:Qcr-consequences} makes $R$ irredundant and $1/6$-stable, since $\lambda_0<1/6$, and supplies Nielsen uniqueness.  Proposition~\ref{prop:Nielsen-Greendlinger}, with parameter $1/6$, gives the assertion.

For part~(6), let $R\in\calQ_{m,q}(n)$ and $S=(s_1,\dots,s_q)\in\calQ_{m,q}(n')$, and suppose $G_R\cong G_S$.  Part~(5) gives $\Phi$, $\sigma$, signs, and conjugators.  Strict minimality of the components gives $n'\ge n$ by applying $\Phi$ to a component of $R$, and $n\ge n'$ by applying $\Phi^{-1}$ to the corresponding component of $S$; hence $n=n'$.  Applying Theorem~\ref{thm:KSS-input}(3) to one matched pair gives a relabeling $\Theta$ with the same outer class as $\Phi$, because the source component has trivial outer stabilizer.  The full family of relations therefore has the form
\[
 s_i=\sh_{t_i}\bigl(\Theta(r_{\sigma(i)})^{\epsilon_i}\bigr).
\]
The converse is immediate.  
\end{proof}

\section{Isomorphism algorithms and enumeration}\label{sec:algorithms-counting}

\subsection{Isomorphism algorithms}\label{subsec:algorithms}

Fix $R=(r_1,\dots,r_q)\in\calQ_{m,q}(n_0)$.  The parameters $m,q,n_0$, the tuple $R$, and the neighborhood data defining $\calQ_{m,q}$ are constants.  For a comparison tuple $S=(s_1,\dots,s_t)$ put
\[
 N=\sum_{j=1}^t|s_j|_A.
\]

\begin{lem}[Checking the comparison promise]\label{lem:comparison-promise}
Fix $q\ge 1$.  Let $S=(s_1,\dots,s_t)$ be a finite tuple of words over $A^{\pm1}$, and let $N$ be their total input length.  In $O(N^2)$ time one can determine whether all of the following hold:
\begin{enumerate}[(1)]
\item $t=q$, and every $s_j$ is nontrivial and cyclically reduced;
\item the tuple $S$ is irredundant;
\item the symmetrization of $S$ satisfies $C'(1/6)$.
\end{enumerate}
\end{lem}

\begin{proof}
A linear scan checks part~\textup{(1)}.  Part~\textup{(2)} is checked by testing each pair for cyclic conjugacy up to inversion, using a doubled word.  Since $q$ is fixed after part~\textup{(1)} has passed, the total time is $O(N^2)$.

For part~\textup{(3)}, form the doubled blocks $s_js_j$ and $s_j^{-1}s_j^{-1}$, separated by distinct fixed markers outside $A^{\pm1}$.  A generalized suffix array and longest-common-prefix data can be built in $O(N\log N)$ time.  For each cyclic relator use only starting positions in the first copy of its doubled block, and cap every longest-common-prefix value at the length of that relator, so that no comparison runs beyond one cyclic turn.  Group starts representing the same member of the symmetrized set and discard repetitions.  There are $O(N)$ remaining starts; checking every pair for a common prefix whose length is at least one sixth of the shorter relator costs $O(N^2)$ in total.  The separating markers prevent comparisons from crossing between blocks.  This condition is exactly the failure of the strict $C'(1/6)$ inequality, so its absence is exactly $C'(1/6)$.
\end{proof}

\begin{lem}[Search output from a fixed strictly minimal word]\label{lem:KSS-search-output}
Let $N\ge 1$, let $u$ be a fixed strictly minimal cyclically reduced word, and suppose that a cyclically reduced word $v$ of length at most $N$ lies in the same $\Out(F_m)$-orbit as $u$.  The search form of Whitehead's algorithm produces
\[
 \Phi=\Theta_k\cdots\Theta_1,
 \qquad k=O(N),
\]
where the $\Theta_i$ are Whitehead automorphisms, $\Phi[u]=[v]$, and every prefix image $\Theta_j\cdots\Theta_1[u]$ has cyclic length $O(N)$.  The search takes $O(N^2)$ time.
\end{lem}

\begin{proof}
By Theorem~\ref{thm:KSS-input}(4), Whitehead's algorithm takes $O(N^2)$ time and records its moves.  Minimize $v$ by length-decreasing Whitehead moves.  There are at most $N$ such moves and all intermediate cyclic lengths are at most $N$.  Since $u$ is strictly minimal and equivalent to $v$, the terminal minimum is obtained from $u$ by a relabeling and a cyclic shift, by Theorem~\ref{thm:KSS-input}(3).  Start with this fixed-length move and reverse the recorded minimization path.  This process gives the stated Whitehead word and prefix-length bound.
\end{proof}

\begin{lem}[Applying a recorded search to a fixed tuple]\label{lem:apply-search-output}
Let $N\ge 1$ and let $u\in F_m\setminus\{1\}$ be fixed.  Suppose that, for some $c>0$,
\begin{equation}\label{eq:fixed-current-lower}
 \left\langle T,\frac{\eta_u}{\|u\|_A}\right\rangle
 \ge c\Lambda_A(T)
 \qquad(T\in\cvbar_m\setminus\{0\}).
\end{equation}
Let $R=(r_1,\dots,r_q)$ be a fixed tuple of nontrivial elements of $F_m$.  Suppose
\[
 \Phi=\Theta_k\cdots\Theta_1,
 \qquad k=O(N),
\]
where the $\Theta_i$ are Whitehead generators.  Put
\[
 \Phi_j=\Theta_j\cdots\Theta_1\in\Aut(F_m),
 \qquad
 \phi_j=[\Phi_j]\in\Out(F_m),
\]
and suppose that every prefix image $\Phi_j[u]$ has cyclic length $O(N)$.  Then cyclically reduced representatives of all $\Phi[r_i]$ can be computed in $O(N^2)$ time.
\end{lem}

\begin{proof}
Taking $T=T_A\phi_j$ in \eqref{eq:fixed-current-lower} gives
\[
 \|\Phi_j(u)\|_A
 \ge c\|u\|_A\Lambda_A(\phi_j),
\]
so $\Lambda_A(\phi_j)=O(N)$.  An equivariant $\Lambda_A(\phi_j)$-Lipschitz map $T_A\to T_A\phi_j$ bounds translation lengths, and hence
\[
 \|\Phi_j(r_i)\|_A
 \le\Lambda_A(\phi_j)\|r_i\|_A
 =O(N)
\]
for every fixed component.  Apply the Whitehead generators successively to cyclically reduced representatives and cyclically reduce after each step.  Since $m$ is fixed, each Whitehead generator replaces every basis letter by a word of uniformly bounded length, and substitution followed by free and cyclic reduction is linear in the current word length.  Thus each stage costs $O(N)$; there are $O(N)$ stages, and $q$ is fixed.
\end{proof}

\phantomsection\label{proof:mainthm-algorithms}
\begin{mainthmrestated}{mainthm:algorithms}{Isomorphism algorithms}
Fix $m\ge 2$ and $q\ge 1$.  For a finite tuple $U=(u_1,\dots,u_s)$ of elements of $F_m$, write
\[
 G_U=\langle a_1,\dots,a_m\mid u_1,\dots,u_s\rangle.
\]
\begin{enumerate}[(1)]
\item Fix $R=(r_1,\dots,r_q)\in\calQ_{m,q}(n_0)$.  There is a quadratic-time algorithm for the following restricted isomorphism problem.  The input is a finite tuple $S=(s_1,\dots,s_t)$ of nontrivial cyclically reduced words, known to be irredundant and to have symmetrization satisfying $C'(1/6)$; the algorithm decides whether $G_R\cong G_S$.  With
\[
 N=\sum_{j=1}^t|s_j|_A,
\]
the running time of the algorithm is $O(N^2)$.  The assumptions on $S$ can be checked within the same $O(N^2)$ time bound; an input failing it is reported as outside the restricted problem.
\item There is a linear-time algorithm which, given tuples $R\in\CR_n^q$ and $S\in\CR_{n'}^q$ known to belong to $\calQ_{m,q}(n)$ and $\calQ_{m,q}(n')$, respectively, decides whether $G_R\cong G_S$.  The algorithm's running time is linear in $n''=\max\{n,n'\}$; the algorithm does not test the assumptions on $R,S$.
\end{enumerate}
\end{mainthmrestated}

\begin{proof}
For part~(1), return a negative answer if $t\ne q$.  If $t=q$, use Lemma~\ref{lem:comparison-promise} to verify nontriviality, cyclic reduction, irredundancy, and the $C'(1/6)$ condition; if this verification fails, report that the input lies outside the promised comparison class.  The first relator $r_1$ is strictly minimal and has trivial outer stabilizer by (Q4).  For every $j$ and $\epsilon\in\{\pm1\}$, run the search form of Whitehead's algorithm on
\[
 [r_1],\ [s_j^\epsilon].
\]
The total cost of these fixed $2q$ searches is $O(N^2)$ by Theorem~\ref{thm:KSS-input}(4).

For each successful search, Lemma~\ref{lem:KSS-search-output} gives a candidate $\Phi$ with a recorded Whitehead word.  Condition (Q2) and \eqref{eq:generic-current-lower} give \eqref{eq:fixed-current-lower} for $u=r_1$.  Lemma~\ref{lem:apply-search-output} computes all $\Phi[r_i]$ in $O(N^2)$ time.  Compare their conjugacy classes, up to independent inversion and permutation, with those represented by the $s_j$.  Because $q$ is fixed, this costs at most $O(N^2)$.

A passing candidate induces an isomorphism.  Conversely, if $G_R\cong G_S$, Theorem~\ref{mainthm:rigidity}(5) supplies an automorphism carrying $[r_1]$ to some $[s_j^\epsilon]$.  The corresponding search succeeds.  Any two candidates with this effect differ by an element of $\Stab_{\Out(F_m)}([r_1])$, which is trivial; hence the candidate found by the algorithm carries the entire relator family correctly.  This proves part~(1).

For part~(2), note that, since $q$ is fixed, $\sum_{r\in R}|r|+\sum_{s\in S}|s|=q(n+n')\le 2qn''=O(n'')$. To prove part~(2), we first compare the common relator lengths $n,n'$.  If $n\ne n'$ the $G_R\not\cong G_S$ by Theorem~\ref{mainthm:rigidity}(6). If $n=n'$ (so that $n''=n=n'$), we test the fixed $2^m m!$ relabelings, the fixed $q!$ component permutations, the $2^q$ inversion choices, and the required cyclic-conjugacy relations to check whether $S$ can be obtained from $R$ as in Theorem~\ref{mainthm:rigidity}(6) and thus determine whether or not $G_R\cong G_S$.  Theorem~\ref{mainthm:rigidity}(6) proves correctness of this procedure.  Since $m,q$ are fixed and cyclic conjugacy is linear-time decidable, the total time is linear in the input length $\sum_{r\in R}|r|+\sum_{s\in S}|s|$ and thus linear in $n''$.  
\end{proof}

\subsection{Enumeration}\label{subsec:counting}

The class $\calQ_{m,q}(n)$ is invariant under $\Gamma_{m,q,n}$.  Indeed, (Q1) uses the invariant core; the neighborhood $\mathcal U$ is relabeling-invariant and counting currents are unchanged by cyclic shift and inversion; signed overlap is invariant under all operations in \eqref{eq:Gamma-action}; and the class $\mathcal Z_m$ is invariant under relabeling, cyclic shift, and inversion, while tuple stabilizers conjugate.  By (Q4), the action on $\calQ_{m,q}(n)$ is free.

Put
\[
 Q_n=\calQ_{m,q}(n),
 \qquad
 B_n=\CR_n^q\setminus Q_n.
\]

\phantomsection\label{proof:mainthm-counting}
\begin{mainthmrestated}{mainthm:counting}{Asymptotic number of isomorphism types}
Fix $m\ge 2$ and $q\ge 1$.  Let $I_{m,q}^{\mathrm{cr}}(n)$ be the number of group isomorphism types represented by presentations
\[
 \langle A\mid r_1,\dots,r_q\rangle,
 \qquad (r_1,\dots,r_q)\in\CR_n^q.
\]
Then
\[
 I_{m,q}^{\mathrm{cr}}(n)
 \sim
 \frac{(2m-1)^{qn}}{2^{m+q}m!\,q!\,n^q}
 \qquad(n\to\infty).
\]
\end{mainthmrestated}

\begin{proof}
Proposition~\ref{prop:Qcr-generic} gives
\begin{equation}\label{eq:cr-good-bad}
 |Q_n|=|\CR_n|^q\bigl(1+O(e^{-cn})\bigr),
 \qquad
 |B_n|=O(|\CR_n|^qe^{-cn}).
\end{equation}
By Lemma~\ref{lem:Gamma-order}, every $\Gamma_{m,q,n}$-orbit in $Q_n$ has size
\[
 2^{m+q}m!\,q!\,n^q.
\]
Theorem~\ref{mainthm:rigidity}(6) says that two tuples in $Q_n$ define isomorphic groups exactly when they lie in the same orbit.  Hence the number of isomorphism types represented in $Q_n$ is
\[
 \frac{|Q_n|}{2^{m+q}m!\,q!\,n^q}.
\]
Every isomorphism type not represented in $Q_n$ has a representative in $B_n$, so
\[
 \frac{|Q_n|}{2^{m+q}m!\,q!\,n^q}
 \le I_{m,q}^{\mathrm{cr}}(n)
 \le
 \frac{|Q_n|}{2^{m+q}m!\,q!\,n^q}+|B_n|.
\]
Since exponential decay dominates $n^{-q}$, the bad-set contribution is $o\left(\frac{|\CR_n|^q}{n^q}\right)$.

Finally, \eqref{eq:number-cyclically-reduced} gives $|\CR_n|\sim(2m-1)^n$, proving the formula.
\end{proof}

\section{Disclosure of AI use}

\noindent Preparation of this paper substantially relied on chats with ChatGPT, but the author verified all the proofs given, reviewed and edited the content as needed and takes full responsibility for the content of the paper.

\end{document}